\documentclass[12pt,article]{amsart}

\usepackage[margin=1in]{geometry}
\usepackage{amsmath,amssymb,amsthm,url,mathrsfs,graphicx,amscd,amsfonts}
\usepackage[mathscr]{eucal}
\usepackage{dsfont}
\usepackage{mathtools}
\usepackage{hyperref}

\allowdisplaybreaks
\newtheorem{theorem}{Theorem}[section]
\newtheorem{lemma}[theorem]{Lemma}
\newtheorem{corollary}[theorem]{Corollary}
\newtheorem{proposition}[theorem]{Proposition}

\newtheorem*{theorem*}{Theorem}
\theoremstyle{remark}\newtheorem*{remark}{Remark}

\numberwithin{equation}{section}

\renewcommand{\pmod}[1]{\left(\mathrm{mod}\,#1\right)}
\renewcommand{\phi}{\varphi}
\newcommand{\fd}{\mathfrak d}
\newcommand{\fa}{\mathfrak a}
\newcommand{\fb}{\mathfrak b}
\newcommand{\fu}{\mathfrak u}
\newcommand{\ul}{\underline}
\newcommand{\fr}{\mathfrak r}
\newcommand{\fc}{\mathfrak c}
\newcommand{\fs}{\mathfrak s}
\newcommand{\fp}{\mathfrak p}
\newcommand{\fq}{\mathfrak q}
\newcommand{\fh}{\mathfrak h}
\newcommand{\OK}{\mathcal{O}_K}

\begin{document}

\title[Bounded gaps between product of two primes ]{Bounded gaps between product of two primes in imaginary quadratic number fields}

\author{Pranendu Darbar}
\address{Indian Statistical Institute   \\
Kolkata  \\
West Bengal 700108,  India}
\email[Pranendu Darbar]{darbarpranendu100@gmail.com}

\author{Anirban Mukhopadhyay}
\address{Institute of Mathematical Sciences\\ CIT Campus, Taramani, Chennai 600 113, India and Homi Bhabha National Institute, Training School Complex, Anushakti Nagar, Mumbai 400 094, India}
\email[Anirban Mukhopadhyay]{anirban@imsc.res.in} 

\author{G.K. Viswanadham}
\address{Department of Mathematics, IISER Berhampur, Berhampur, Odisha 760 010, India}
\email[G.K. Viswanadham]{vissu35@gmail.com} 



\begin{abstract}
We study the gaps between products of two primes in imaginary quadratic number 
fields using a combination of the methods of Goldston-Graham-Pintz-Yildirim 
\cite{GGPY}, and Maynard \cite{MAY}. 
An important consequence of our main theorem is existence of infinitely many 
pairs $\alpha_1, \alpha_2$ which are product of two primes in the 
imaginary quadratic field $K$ such that $|\sigma(\alpha_1-\alpha_2)|\leq 2$ 
for all embedding $\sigma$ of $K$ if the class number of $K$ is one 
and $|\sigma(\alpha_1-\alpha_2)|\leq 8$ for all embedding $\sigma$ of $K$ 
if the class number of $K$ is two.
\end{abstract}

  File \jobname.tex.\par

\maketitle

\section{Introduction}
One of the classical problems in prime number theory is to study the gaps between prime numbers. The famous twin prime conjecture asserts that there are infinitely many pairs $(p_1, p_2)$ of prime numbers such that $|p_1-p_2|=2$. Although this conjecture remains  out of reach, study of this conjecture leads to several interesting results. The first and very important breakthrough in this direction is the result of Goldston et. al \cite{GPY} who showed that
\begin{equation*}
 \liminf_{n\rightarrow \infty}\frac{p_{n+1}-p_n}{\log p_n}=0\; ,
\end{equation*}
 where $p_n$ denotes the $n$th prime.  Zhang~\cite{ZHN} has subsequently improved this result  by showing that
 \begin{equation*}
  \liminf_{n\rightarrow \infty}(p_{n+1}-p_n)\le 7\times 10^7.
\end{equation*}
 Zhang's results was sharpened and simplified by Mayanard~\cite{MAY}.  
 The Polymath 8b group \cite{POLY}, led by Tao, subsequently proved 
 that Zhang's result holds with 246 instead of $7\times 10^7$.  
 It is proved by Goldston et.al. \cite{GPY} that 246 can be replaced 
 by 16 if one assumes Elliott-Halberstam conjecture.

 \vspace{2mm}
 \noindent
 Let $\{q_n\}_{n\geq 1}$ be the sequence of positive integers which are products of exactly two primes written in the increasing order. The members of this sequence are called $E_2$ numbers. Heuristically problems involving $E_2$ numbers are as difficult as problems involving prime numbers as sieve methods do not seem to distinguish between numbers with even numbers of prime factors and odd number of prime factors (parity principle of Selberg~\cite{selberg}). Hence it is interesting to study the gaps between $E_2$- numbers. This study was initiated by Goldston et.al \cite{GGPY0} who showed that
 \begin{equation}
 \label{gap26}
 \liminf_{n}(q_{n+1}-q_n)\leq 26\; .
 \end{equation}
 Later developing the methods of \cite{GGPY0}, they are able to improve the constant on the right hand of \eqref{gap26} to $6$ in \cite{GGPY}.

\vspace{2mm}
\noindent
Let $K$ be a number field and let $\mathcal{O}_K$ be its ring of integers. We say that an element $\alpha\in \mathcal{O}_K$ is prime if the principle ideal $\alpha\mathcal{O}_K$ is a prime ideal.  Castillo et.al. \cite{CAS} has initiated the study of gaps between primes in number fileds. By extending the methods of Maynard-Tao they showed that for a totally real field $K$ there are infinitely many primes $\alpha_1$ and $\alpha_2$ in $\mathcal{O}_K$ such that $|\sigma(\alpha_1-\alpha_2)|\leq 600$ for every embedding $\sigma$ of $K$. The case when $K$ is imaginary is first considered by Vatwani~\cite{VAT}. In particular it is shown in \cite{VAT} that there are infinitely many pairs $({\bf p}_1, {\bf p}_2)\in \mathbb{Z}[i]\times \mathbb{Z}[i]$ such that ${N}({\bf p}_1-{\bf p}_2)< 246^2$, where ${N}(\cdot)$ denotes the norm function on $\mathbb{Q}(i)$. The method of the proof can be generalized to cover any imaginary quadratic number field with class number 1.

\vspace*{2mm}
\noindent
In the spirit of \cite{CAS, GGPY0, GGPY} it is natural consider the gaps between product of two primes in number fields. Before stating our main result of this article, we will fix some notations. Let $\mathcal P$ be the set of prime numbers in $\mathcal{O}_K$. Let $G_2^K$ be the set of all $\alpha\in \mathcal{O}_K$ which  can be written as product of two elements from $\mathcal P$.
We say that a tuple $(\mathfrak{h}_1,\dots,\mathfrak{h}_k)\in \mathcal{O}_K^k$ is admissible if it does not cover all the residue classes modulo $\mathfrak{p}$ for any prime ideal $\mathfrak{p}$ of $\mathcal{O}_K$. 

\vspace{2mm}
\noindent
Now we are in a position to state the main result of this paper.
\begin{theorem}\label{main theorem}
Let  $K$ be an imaginary quadratic number field and  let $r\geq 2$ be an integer. Then there exists a positive integer $\tilde{k}:=\tilde{k}(r, K)$ such that for any admissible $k$-tuple $(\mathfrak{h}_1, \ldots, \mathfrak{h}_k)\in \mathcal{O}_K^k$ with $k\geq \tilde{k}$, there are infinitely many $\alpha \in \mathcal{O}_K$ such that at least $r$ of $\alpha + \mathfrak{h}_1, \ldots, \alpha+\mathfrak{h}_k$ are $G_2^K$-numbers. 
\end{theorem}
 \noindent
 It is clear from Theorem 1.1 that $\liminf |\sigma(\alpha-\beta)|\leq M(K)$
where $M(K)$ is a constant depends only on $K$ and the $\liminf$ is taken when $\alpha,\beta$ runs over all $G_2^K$ numbers. It will be clear at the end of the proof that the constant depends only on the class number. In the
following corollaries we will precisely give the value of $M(K)$ when the
class number is 1 or 2.
\begin{corollary}\label{corollary: gap 2 number field}
Let $K_d:=\mathbb{Q}\left(\sqrt d\right)$ be imaginary quadratic field with
class number one $($ there are exactly nine such fields corresponding to 
$d=-1, -2, -3, -7, -11, -19, -43, -67$ and $-163 )$.  
There exist infinitely many $G_2^{K_d}$-numbers $\alpha_1, \alpha_2$ 
such that $|\sigma(\alpha_1-\alpha_2)|\leq 2$ for all embeddings 
$\sigma$ of $K_d.$
\end{corollary}

\begin{remark}
For $d=-1$ and $-2$, we consider the admissible pair $\{0,2\}$.
Then there are infinitely many rational primes of the form 
$p_1=a^2+db^2, p_2=m^2+dn^2, p_3=a_1^2+db_1^2, p_4=m_1^2+dn_1^2$ with $p_1p_2=(a^2+db^2)(m^2+dn^2)$ and $p_3p_4=(a^2+db^2)(m^2+dn^2)+2h\begin{vmatrix}
a & \sqrt{d}b \\ 
\sqrt{d}n & m \\
\end{vmatrix}+h^2$ 
with $a, b, m, n, a_1, b_1, m_1, n_1\in \mathbb{Z}$ and $|h|\leq 2.$

Similarly for $d=-3, -7, -11, -19, -43, -67$ and $-163$ 
considering the admissible pair $\{0,2\}$ we get
infinitely many rational primes of the form 
$p_1=(a^2+db^2)/4, p_2=(m^2+dn^2)/4, 
p_3=(a_1^2+db_1^2)/4, p_4=(m_1^2+dn_1^2)/4$ with 
$p_1p_2=(a^2+db^2)(m^2+dn^2)/16$ and 
$p_3p_4=\frac{1}{16}(a^2+db^2)(m^2+dn^2)+\frac{1}{16}\left(2h
\begin{vmatrix}
a & \sqrt{d}b \\ 
\sqrt{d}n & m \\
\end{vmatrix}
+h^2\right)$ 
with $a, b, m, n, a_1, b_1, m_1, n_1\in \mathbb{Z}$ and $|h|\leq 2.$
\end{remark}

\begin{corollary}\label{corollary: gap 8 number fields}
Let $K_d:=\mathbb{Q}\left(\sqrt d\right)$ be imaginary quadratic field with
class number two.  
There exist infinitely many $G_2^{K_d}$-numbers $\alpha_1, \alpha_2$ 
such that $|\sigma(\alpha_1-\alpha_2)|\leq 8$ for all embeddings 
$\sigma$ of $K_d$.
\end{corollary}

This article is organized as follows. 
In Section~\ref{prelims} we provide the necessary preliminaries to 
prove Theorem~\ref{main theorem}. 
In Section~\ref{sec:bombieri} we prove a variant of Bombieri-Vinogradov theorem 
for $G_2^K$-numbers. 
In Section~\ref{method} we explain the method of the proof.  
In Section~\ref{preparations} we will prove some preparatory lemmas 
which are essential for the proof. 
In Section~6 we will choose the appropriate weights. 
Section~\ref{prop:main} is devoted to prove Proposition~\ref{prop:main}. 
In Section~\ref{proofs} we will conclude the proofs of 
Theorem~\ref{main theorem}, Corollary~\ref{corollary: gap 2 number field}
and Corollary \ref{corollary: gap 8 number fields}.
 
 \section{Notations and preliminaries} \label{prelims}

Here and in what follows, $K$ denotes an imaginary quadratic  field unless otherwise mentioned. For much of this article, we follow the  notations of Hinz \cite{HIN1} and Castillo et.al 
\cite{CAS}. 
Being an imaginary quadratic field $K$ has no real embeddings and it has exactly two complex embeddings, namely $\sigma_0$ (the identity) and 
$\sigma$ (complex conjugation). We observe that for any non-zero  $\alpha \in \mathcal{O}_K$, 
$|\sigma(\alpha)|\ge 1.$ For $N>1$, let 
\begin{equation*}
A^0(N)=\{\alpha\in \mathcal{O}_K : 1\le |\sigma(\alpha)|\le N \}
\text{ and }
\mathcal{P}^0(N)=\mathcal{P}\cap A^0(N).
\end{equation*}
Further, for $N_1<N_2$, we define 
\[
A(N_1, N_2)=A^0(N_2)\setminus A^0(N_1), \ \
\mathcal{P}(N_1,N_2)=A(N_1, N_2)\cap \mathcal{P}.
\] 
We would also use
$A(N)$ and $\mathcal{P}(N)$ for$A(2N,N)$ and $\mathcal{P}(2N,N)$ 
respectively. For a set $S$, $|S|$ denotes its cardinality, 
for an element $\alpha\in K$ and an ideal $\fq$ of $\OK$,
$|\alpha|$ and $|\fq|$ denotes the respective norms. 

\begin{remark}
 A clarification about the notations is much called for at this point. For an 
 element $\alpha\in\OK$, $|\alpha|$ denotes its norm whereas $|\sigma(\alpha)|$
 denotes absolute value as a complex number. For imaginary quadratic fields, 
 they are related by 
 $$|\alpha|=\sigma_0(\alpha)\sigma(\alpha)=|\sigma(\alpha)|^2.$$
 Hence $A^0(N)$ as defined above can also be described as
 \[
A^0(N)=\{\alpha\in \mathcal{O}_K : 1\le |\alpha|\le N^2 \}.
\]
These usages will be clear from the context as we proceed.
\end{remark}

For elements $a,b\in \mathcal{O}_K$ and an ideal $\fq$ of $\mathcal{O}_K$,
we write $a\equiv b \mod \fq$ to mean that the ideal generated by $a-b$
is contained in $\fq$, i.e $(\fa-\fb)\subset \fq$. 
Moreover, if ideals $(a)$ generated by $a\in \mathcal{O}_K$ 
and $\fq$ does not have common factor, then we write $(a,\fq)=1$

\vspace{2mm}
\noindent
Given a non-zero ideal $\mathfrak{q}\subseteq \OK$, 
we define analogues of three classical multiplicative functions, namely the norm 
$|\mathfrak{q}|:=|\OK/\mathfrak{q}|$, the Euler {\it phi-function} 
$\varphi(\mathfrak{q}) :=|(\OK/\mathfrak{q})^\times|$ and the M\"obius function 
$\mu(\mathfrak{q}):=(-1)^r$ if $\mathfrak{q}=\mathfrak{p}_1\dots\mathfrak{p}_r$ for 
distinct prime ideals $\mathfrak{p}_1,\dots,\mathfrak{p}_r$ and $\mu(\mathfrak{q})=0$ otherwise.
We use $\tau_k(\fq)$ to denote the number of ways of writing $\fq$ as a product of $k$ factors
and $\omega(\fq)$ to denote the number of distinct prime ideals containing $\fq$. For ideals $\fa, \fb$, we use $[\fa,\fb]$ and $(\fa,\fb)$ to denote
LCM and GCD of $\fa,\fb$.

\vspace{2mm}
\noindent
The $k$-tuple $(\fa_1, \ldots, \fa_k)$ with $\fa_j\in \OK$ for all $j \; (1\leq j\leq k)$ is denoted by $\ul\fa.$ 
We use $w_1, w_2$ to denote prime elements of $\mathcal{O}_K.$ 
For any $R\in \mathbb{R}$, $|\ul\fa|\leq R$ is to be interpreted as $\prod_{j=1}^{k}|\fa_j|\leq R.$
\noindent
The notion of divisibility among $k$-tuples is defined componentwise, i.e, 
\[
\ul\fa|\ul\fb \Leftrightarrow \fa_j|\fb_j \; \; \; \forall 1\leq j \leq k.
\]
For any integral ideal 
$\fq$ of $\mathcal{O}_K, \ul\fa| \fq \Leftrightarrow \prod_{j=1}^{k}\fa_j |\fq.$ 
We use the notation $[\ul\fa, \ul\fb]$ to denote the product of the component-wise 
least common multiples, i.e. $[\ul\fa,\ul\fb]=\prod_{j=1}^{k}[\fa_j, \fb_j]$ 
and $(\fa,\fb)=1$ to mean that the ideals $\fa$ and $\fb$ are coprime, 
where $1$ is the trivial ideal.

\vspace{2mm}
\noindent
For $\mathcal{R}e(s)>1,$ the Dedekind zeta function of $K$ is defined by
\[
\zeta_{K}(s) := \sum_{\mathfrak{q}\subseteq \mathcal{O}_K} |\mathfrak{q}|^{-s}
\]
where sum is over all non-zero ideals of $\mathcal{O}_{K}.$ 
This function admits meromorphic continuation to the whole complex plane with 
a pole at $s=1$. Let $c_K$ denote its residue at $s=1$.
\vspace{2mm}
\noindent
 Now we note that \cite[page 4]{CAS} the number of elements $\alpha\in A(N)$ satisfying a congruence condition 
 $\alpha \equiv \alpha_0 \pmod{\mathfrak{q}}$ is given by
\[
\frac{|A(N)|}{|\mathfrak{q}|} + O(|\partial A(N,\mathfrak{q})|),
\]
where
\[
|\partial A(N,\mathfrak{q})| \ll
1+\left(\frac{|A(N)|}{|\mathfrak{q}|}\right)^{\frac{1}{2}}. 
\]

\vspace{2mm}
\noindent
The following lemma is central in estimation of the sums that arise in Selberg's higher dimensional sieve.

\begin{lemma}[Lemma 2.5, \cite{CAS}] \label{lem:Dimensional sieve}
Suppose $\gamma$ is a multiplicative function on the non zero ideals of $\OK$ such that there are constants $\kappa >0, A_1 > 0$, $A_2 \ge 1$, and $L \ge 1$ satisfying
\[
0 \leq \frac{\gamma(\mathfrak{p})}{|\mathfrak{p}|} \leq 1 - A_1,
\]
and
\[
-L \leq
\sum_{w \leq |\mathfrak{p}|< z}\frac{\gamma(\mathfrak{p}) \log |\mathfrak{p}|}{|\mathfrak{p}|} - \kappa \log (z/w)
\leq A_2,
\]
for any $2 < w \leq z$. Let $h$ be the completely multiplicative function defined on prime ideals by 
$h(\mathfrak{p}) = \gamma(\mathfrak{p}) / (|\mathfrak{p}| - \gamma(\mathfrak{p}))$. 
Let $G\colon [0,1] \to \mathbb{R}$ be a piecewise differentiable function and 
let $G_{max} = \sup_{t \in [0,1]}(|G(t)| + |G'(t)|)$. Then
\[
\sum_{|\mathfrak{d}| < z} \mu(\mathfrak{d})^2 h(\mathfrak{d}) G \left(\frac{\log |\mathfrak{d}|}{\log z} \right)
=
\mathfrak{S}\frac{c_K^\kappa (\log z)^\kappa}{\Gamma(\kappa)} \int_0^1 G(x)x^{\kappa - 1} dx
+ O_{K,A_1, A_2, \kappa} \left(L G_{max} (\log z)^{\kappa - 1}
\right),
\]
where $c_K := \mathop{\mathrm{Res}}_{s=1} \zeta_K(s)$ and the singular series
\[
\mathfrak{S} =
\prod_{\mathfrak{p}} \left(1 - \frac{\gamma(\mathfrak{p})}{|\mathfrak{p}|} \right)^{-1}
\left(1 - \frac{1}{|\mathfrak{p}|} \right)^\kappa.
\]
\end{lemma}

\vspace{2mm}
\noindent

The following lemma is a consequence of Minkowski's lattice point theorem 
(see \cite[page 12]{CAS}).
\begin{lemma}\label{lem:size of A(N)}
Let $A^{0}(N)$ and $A(N)$ be defined as above. We have
 \[ 
  |A^0(N)|=(1+o(1)) \frac{2\pi N^2}{\sqrt{|D_K|}}
\  {\text and } \
  |A(N)| =(1+o(1))\frac{6\pi N^2}{\sqrt{|D_K|}}
 \]
where $D_K$ is the discriminant of $K$.
\end{lemma}

Let $\omega_K$ be the number of roots of unity contained in $K$ and $h_K$ be the
class number of $K$.
The following theorem is a special case of Mitsui's generalized Prime number theorem \cite{MIT}.
\begin{lemma}\label{lem:PNT1}
Let $\mathcal{P}^{0}(N)$ be defined as above. We have
\[
 \mid \mathcal P^0(N)\mid
 =\frac{\omega_K}{h_K R_K}
 \int_2^{N^2}
 \frac{du}{\log u}
 +O_K( N^2 \exp(-c\sqrt{\log N}))
\]
where $c$ is a non-zero positive real number.
\end{lemma}
\noindent
We denote $m_K:=\frac{\omega_K}{h_K R_K}$ as Mitsui's constant.
As a direct consequence of Lemma \ref{lem:PNT1} we get
\begin{lemma}\label{lem:PNT2} Let $\mathcal{P}^0(N)$ be defined as above. 
Then we have
\[ 
|\mathcal{P}^0(N)| =\frac{\omega_K}{h_K} \frac{N^2}{\log (N^2)}
\left(1+O\left(\frac{1}{\log N}\right)\right).  
\]
\end{lemma}

The following lemma is Dedekind's class number formula.
\begin{lemma}[\cite{Neukirch}, Corollary 5.11 ]
\label{lem:CNF} 
Let $c_K, \omega_K$ and $ h_K$ be defined as above. We have
\[ c_K = \frac{2\pi h_K}{\omega_K \sqrt{|D_K|}}. \]
\end{lemma}

\begin{lemma}\label{lem:merten's theorem on number field}
Let $K$ be an algebraic number field. For any natural number $R\ge 2$, 
we have
\[
\sum_{\substack{\fu\subseteq \mathcal{O}_K \\ |\fu|\leq R}}\frac{1}{\mid\fu\mid} \ll_{K} \log R
\quad \text{ and } \quad
\sum_{\substack{\fp\in \mathcal{P} \\ |\fp|\leq R}}\frac{1}{|\fp|}\ll_{K} \log \log R ,
\]
where first sum is over all non-zero integral ideals of $\OK$ whose norm is less than or equal to R.
\end{lemma}

\section{A generalization of Bombieri-Vinogradov theorem}
\label{sec:bombieri}
A subset $S$ of $\OK$ is said to have {\it level of distribution} $\vartheta$ for $0<\vartheta\le 1$ if for any $C>0$
there exists a constant $B=B(C)$ such that
\begin{equation}\label{eq:1.17}
\sum_{|\fq|\leq \frac{|A^0(N)|^{\vartheta}}{(\log |A^0(N)|)^{B}}} 
\max_{M\le N}
\max_{\substack{a\\ (a,\fq)= 1}} 
\Bigg|  \sum_{\substack{w\in S\cap A^0(M) \\ w \equiv a \pmod \fq}} 
1 -\frac{|S\cap A^0(M)|}{\varphi(\fq)}\Bigg|
   \ll_{A,K} \, \frac{|A^0(N)|}{(\log N)^C}\; ,
\end{equation}
\noindent
An analog of Elliott-Halberstam conjecture for number fields predicts that 
the inequality (\ref{eq:1.17}) holds with 
 any $\vartheta$ in $0<\vartheta \le 1$.
Hinz \cite{HIN1} showed that primes have level of distribution $\frac{1}{2}$  in totally real algebraic number fields. 
Huxley \cite{ HUX} obtained level of distribution $\frac{1}{2}$ for an weighted version of 
(\ref{eq:1.17}).  
The $G_2^K$-numbers for $K=\mathbb Q$ was shown by Motohashi to have level of distribution $\frac{1}{2}$.
 
\vspace{2mm}
\noindent
For our purposes, it is convenient to define the following related quantities. 
\begin{align*} 
 &\pi^\flat (N)=|\mathcal{P}^0(N)|, \; \; \; \; \; 
 \pi^\flat (N;\mathfrak{q},\fa) = \sum_{\substack{ w\in \mathcal{P}^0(N)\\
  w\equiv \fa \pmod {\mathfrak{q}}}} 1 ,\; \;\; \; \; \\
&\varepsilon(N;\mathfrak{q},\alpha) = \pi^\flat(N;\mathfrak{q},\alpha) - \frac{1}{\phi(\mathfrak{q})} \pi^\flat(N), \; \; \;
\varepsilon^*(N;\mathfrak{q}) = \max_{M\le N} \max_{\alpha; (\alpha,\mathfrak{q})=1} |\varepsilon(M;\mathfrak{q},\alpha)|\; .
\end{align*}
 
Using a theroem of \cite{HUX} and following the argument in Lemma $10.2$ of \cite{VAT}, we prove the following generalization of Bombieri-Vinodradov theorem.
\begin{proposition}\label{Bombieri vinogradov theorem on number fields}
Let $K$ be an imaginary quadratic number field. Then \eqref{eq:1.17} 
holds for any  $\vartheta\leq \frac{1}{2}$ when $\mathcal{S}=\mathcal{P}$.
\end{proposition}
\begin{proof}
Let $\fq$ be an ideal in $\mathcal{O}_K$. We denote the ray class group 
$\pmod \fq$ by $\mathcal{C}_{\fq}$ and a ray class by $\mathcal{L}_{\fq}$. Let $\pi(x, K)$ be the number of prime ideals in $\mathcal{O}_K$ of norm $\leq x$ and $\chi_{\mathbb{P}}$ be the characteristic function of the prime ideals in $\mathcal{O}_K$. We define
\[
E(x, \fq, \mathcal{L}_{\fq})=\displaystyle\sum_{\substack{\fa \subset \mathcal{O}_K\\ |\fa|\leq x\\ \fa \in \mathcal{L}_{\fq}}}\chi_{\mathbb{P}}(\fa)-\frac{\pi(x, K)}{h(\fq)},
\]
where $h(\fq)$ denotes the cardinality of the ray class group $\mathcal{C}_{\fq}$.
\begin{lemma}[Huxley \cite{HUX}]\label{Huxley theorem}
Using the notations as above, for any $A>0$, there exist a real number $B>0$ such that for any $\vartheta \leq \frac{1}{2}$ we have
\[
\displaystyle\sum_{|\fq|\leq \frac{x^{\vartheta}}{(\log x)^B}}\frac{h(\fq)}{\phi(\fq)}\max_{\mathcal{L}_{\fq}\in \mathcal{C}_{\fq}}\max_{y\leq x}\left|E_{\mathbb{P}}(y, \fq, \mathcal{L}_{\fq})\right|\ll \frac{x}{(\log x)^A}.
\]
\end{lemma}
\noindent
We have the following  relation between the number of ray classes and 
the class number (\cite{VAT}):
\[
h(\fq)=\frac{\phi(\fq)h_K}{[U:U_{\fq, 1}]}
\]
where $U$ is the unit group of $\mathcal{O}_K$, $U_{\fq, 1}
=\{\alpha\in U : \alpha\equiv 1 \pmod \fq, \alpha \succ 0\}$ and 
$h_K$ is the class number of $K$ where $\alpha\succ 0$ means all the 
real conjugates (if any) of $\alpha$ are positive.

\vspace{2mm}
\noindent
Now we will estimate the index set $[U: U_{\fq, 1}]$. To do that we define
the following homomorphism
\[
\psi : U \to \left(\mathcal{O}_K/ \fq\right)^{*}
\]
by $\psi(u)=u \pmod \fq$. Then the kernel of $\psi$ is $U_{\fq, 1}$ and 
image of $\psi$ is the residue classes $\pmod \fq$ that contain a unit. 
Let $T_{\fq}=\text{Im}(\psi)$. Then $|T_{\fq}|=[U: U_{\fq, 1}]$ and 
$\frac{h(\fq)}{\phi(\fq)}=\frac{h}{|T_{\fq}|}$. Since number of units in a imaginary 
quadratic number field is $2, 4$ or $6$, so if $\fu_1, \fu_2 \in U$ satisfies 
$\fu_1\equiv \fu_2 \pmod \fq$ then $|\fq|$ must divide $|\fu_1-\fu_2|$, 
which is atmost $4$. Thus for $|\fq|>4$ we see that $T_{\fq}=|U|$, which only depends only on $K$ and not on $\fq$. 
Therefore using these estimates, from Lemma \ref{Huxley theorem} 
we obtain the following.
\begin{lemma}\label{truncated huxley theorem}
Using the notation as in Lemma \ref{Huxley theorem}, 
for any $A>0$ there exist a positive real number $B$ such that for any 
$0 < \vartheta \leq \frac{1}{2}$, we have
\begin{align}\label{distribution of truncated huxley theorem}
\displaystyle\sum_{4<|\fq|\leq \frac{x^{\vartheta}}{(\log x)^B}}
\max_{\mathcal{L}_{\fq}\in \mathcal{C}_{\fq}}\max_{y\leq x}
\left|E_{\mathbb{P}}(y, \fq, \mathcal{L}_{\fq})\right|\ll_K \frac{x}{(\log x)^A}.
\end{align}
\end{lemma}
\noindent
\textbf{Deduction of Proposition \ref{Bombieri vinogradov theorem on number fields} from Lemma \ref{truncated huxley theorem}}. Let $\fa \in \mathcal{O}_K$, $(\fa, \fq)=1$ and $\mathcal{L}_{\fq}(\fa)$ be the ray class containing $(\fa)$. Then from \eqref{distribution of truncated huxley theorem} we get
\begin{align}\label{distribution of truncated huxley theorem 2}
\displaystyle\sum_{4<|\fq|\leq \frac{x^{\vartheta}}{(\log x)^B}}\max_{(\fa, \fq)=1}\max_{y\leq x}\left|E_{\mathbb{P}}(y, \fq, \mathcal{L}_{\fq}(\fa))\right|\ll \frac{x}{(\log x)^A}.
\end{align}
It is easy to see that all integral ideals belonging to $\mathcal{L}_{\fq}(\fa)$ are 
principal. Therefore we obtain
\[
\sum_{\substack{\fa \subset \mathcal{O}_K\\ |\fa|\leq y\\ \fa \in \mathcal{L}_{\fq}}}\chi_{\mathbb{P}}(\fa)=\sum_{\substack{\eta \in \mathcal{O}_K\\ |\eta|\leq y\\ (\eta) \in \mathcal{L}_{\fq}(\fa)}}\chi_{\mathbb{P}}(\eta).
\]
We also observe that there is an one to one correspondence between
\[
\{ \eta \in \mathcal{P}, |\eta|\leq x, (\eta)\in \mathcal{L}_{\fq}(\fa)\} 
\quad \text{ and } \quad \{w\in \mathcal{P}, |w|\leq x, w\equiv \fa (\fq)\}.
\]
Thus we have
\[
\sum_{\substack{\eta \in \mathcal{O}_K\\ |\eta|\leq y\\ (\eta) \in \mathcal{L}_{\fq}(\fa)}}
\chi_{\mathbb{P}}(\eta)= \displaystyle\sum_{\substack{w \in \mathcal{O}_K\\ |w|\leq y\\ w\equiv \fa (\fq)}}\chi_{\mathbb{P}}(w).
\]
For $|\fq|>4$, we recall that $h(\fq)=\frac{h_K\phi(\fq)}{|U|}$. So from \eqref{distribution of truncated huxley theorem 2}, we get
\begin{align}\label{distribution of truncated huxley theorem 3}
\displaystyle\sum_{4<|\fq|\leq \frac{x^{\vartheta}}{(\log x)^B}}\max_{(\fa, \fq)=1}\max_{y\leq x}\bigg|\displaystyle\sum_{\substack{w \in \mathcal{O}_K\\ |w|\leq y\\ w\equiv \fa (\fq)}}\chi_{\mathbb{P}}(w)-\frac{|U|\pi(y, K)}{h_K\phi(\fq)}\bigg|\ll \frac{x}{(\log x)^A}
\end{align}
for any $\vartheta\leq \frac{1}{2}$ and for any $A>0$. Now Prime ideal theorem tells us
\begin{align}\label{prime ideal theorem}
\pi(y, K)\sim \frac{y}{\log y}
\end{align}
Also from Lemma \ref{lem:PNT2} and using $\omega_K=|U|$, we get
\begin{align}\label{lem: PNT 3}
|\mathcal{P}^{0}(y^{1/2})|\sim \frac{|U|}{h_K}\frac{y}{\log y}.
\end{align}
Combining \eqref{prime ideal theorem} and \eqref{lem: PNT 3} we obtain 
\begin{align}\label{relation between prime and principal ideal}
\pi(y, K)\sim \frac{h_K}{|U|}|\mathcal{P}^{0}(y^{1/2})|.
\end{align}
Also note that 
\[
 \sum_{\substack{ |w|\le y \\ w\equiv \fa(\fq)}}\chi_{\mathbb{P}}(w)
 =\sum_{\substack{ w\in \mathcal{P}^0(y^{1/2}) \\ w\equiv \fa(\fq) }} 1.
\]
From \eqref{relation between prime and principal ideal} and 
\eqref{distribution of truncated huxley theorem 3} we complete proof of the proposition.
\end{proof}

We would use the above result in the following form which can be easily deduced
by partial summation.
\begin{lemma} \label{L:Lemma3}
Let $K$ be an imaginary quadratic number field. 
For any $\vartheta$, $0<\vartheta \le \frac{1}{2}$,
 any $B>0$ and a fixed integer $h\ge 0$,
 there exists $C=C(B,h)$ such that 
 if $Q\le |A(N)|^{\vartheta} (\log N)^{-C}$, then
\begin{equation*}
\sum_{|\mathfrak{q}|\le Q}  \mu^2(\mathfrak{q}) h^{\omega(\mathfrak{q})}  \varepsilon^*(N;\mathfrak{q})  \ll_{B,K} |A(N)| (\log N)^{-B}.
\end{equation*}
\end{lemma}
\noindent

\vspace{2mm}
\noindent
For $0 < \frac{\vartheta}{2}<b\leq \frac{1}{2},$ and for $1 \le Y'\le N^b$ ( $Y'$ to be chosen later )
we define a function $\beta$ on $\OK$ by 
\[ \beta(\alpha)=\begin{cases}
                  1 & \text{ if } \alpha=w_1w_2, w_1\in \mathcal P(Y',N^b),
                                               w_2\in \mathcal P(N^b,\infty)\\
                  0 & \text{ otherwise. }                             
                 \end{cases}
\]
For the function $\beta$, we define 
 \begin{align*}
 &\pi_\beta(N)= \sum_{\alpha\in A(N)} \beta(\alpha), \; \;  \;
\pi_{\beta,\fq}(N)  = \sum_{\substack{ \alpha\in A(N) \\ (\alpha,\fq)=1}} \beta(\alpha),
\; \; \; 
 \pi_\beta(N;\mathfrak{q},\gamma)  = \sum_{\substack{ \alpha\in A(N) \\ \alpha \equiv \gamma \pmod {\mathfrak{q}}}} \beta(\alpha)\\
& \varepsilon_\beta(N;\mathfrak{q},\gamma) = 
 \pi_\beta(N;\fq,\gamma)- \frac{1}{\phi(\fq)} \pi_{\beta,\fq} (N), \; \; \;
 \varepsilon^*_\beta(N;\fq) =  \max_{M\le N} \max_{\gamma; (\gamma,\fq)=1} |\varepsilon_\beta (M;\fq,\gamma)|.
\end{align*}

An arithmetic functions $f$ is said to have {\it level of distribution} $\vartheta$ for $0<\vartheta\le 1$ if for any $A>0$
there exists a constant $B=B(A)$ such that
\begin{equation}
\sum_{q\leq \frac{N^{\vartheta}}{(\log N)^{B}}}\max_{M\leq N} \max_{\substack{a\\ (a,q)= 1}} 
\Bigg|  \sum_{\substack{n\leq M \\ n \equiv a \pmod q}} f(n) -\frac{1}{\varphi(q)}\sum_{\substack{n\leq M \\ (n, q)=1}}f(n)\Bigg|
   \ll_{A} \, \frac{N}{(\log N)^A}.
\end{equation}

Let $\tau(n)$ be the number of divisors of a natural number $n$.
A complex valued arithmetic function $f$ is said to satisfy Siegel-Walfisz condition if there exist positive constant $C$ such that
\begin{align}\label{siegel walfisz}
f(n)=O\left(\tau(n)^C\right) \quad 
\text{and} \quad \sum_{n\leq x}f(n)\chi(n)=O\bigg(\frac{x}{(\log x)^{3D}}\bigg),
\end{align}
holds for all $D>0$ and for any non-principal Dirichlet character $\chi\pmod q$ with $q\ll (\log x)^{D}.$ 

If arithmetic function $f$ and $g$ both satisfy \eqref{siegel walfisz}  and have level of distribution $\frac{1}{2}$ then 
Motohashi \cite{MOT} obtained that the Dirichlet convolution $f*g$ also does so. 
In \cite{DAR}, we extend Motohashi's \cite{MOT} result to arithmetic functions on imaginary quadratic number
fields. As the proof can be carried forward for any level of distribution
 $0<\vartheta\le \frac{1}{2}$, viewing $\beta$ as a Dirichlet convolution of characteristic functions of 
$\mathcal{P}(Y, N^b)$ and $\mathcal{P}(N^b, \infty)$, we get 

\begin{lemma} \label{L:Lemma4}
Let $K$ be an imaginary quadratic number field. 
For $0<\vartheta \le \frac{1}{2}$, $B>0$ and fixed integer $h\ge 0$,
 there exists $C=C(B,h)$ such that  if  
 $Q\le |A(N)|^{\vartheta} (\log N)^{-C}$, then
\begin{equation} \label{E:BVE2}
\sum_{|\mathfrak{q}|\le Q}  \mu^2(\mathfrak{q}) h^{\omega(\mathfrak{q})} \varepsilon^*_\beta(N;\mathfrak{q})   \ll_{B,K} |A(N)| (\log N)^{-B}.
\end{equation}
\end{lemma}

\section{method}
\label{method}

Now we will describe the method of proof which is a combination of methods of \cite{GGPY} and \cite{MAY}.

\vspace{2mm}
\noindent
Recall that a tuple $(\mathfrak{h}_1,\dots,\mathfrak{h}_k)\in \mathcal{O}_K^k$ is admissible if it does not cover all 
residue classes modulo $\mathfrak{p}$ for any prime ideal $\mathfrak{p}$ of 
$\mathcal{O}_K$.  
Let $D_0=\log\log\log N$, $\mathfrak{m}:=\prod_{|\mathfrak{p}|<D_0}\mathfrak{p}$.
Since $(\mathfrak{h}_1,\dots,\mathfrak{h}_k)\in \mathcal{O}_K^k$ is admissible, there
exists $v_0$ modulo $\mathfrak{m}$ 
such that each $\alpha+\fh_i$ lies in $\left(\OK/\mathfrak{m}\right)^\times$ for all
$j=1,\cdots, k$.
The main objects of consideration are the sums
\[
S_1 := 
\sum_{\begin{subarray}{c} \alpha \in A(N) \\ \alpha \equiv v_0\pmod{\mathfrak{m}} \end{subarray}} 
\bigg(\sum_{\begin{subarray}{c} \fd_1,\dots,\fd_k: \\ \fd_i \mid (\alpha+\mathfrak{h}_i)\,\forall i \end{subarray}} 
\lambda_{\fd_1,\dots,\fd_k} \bigg)^2
\]
and 
\begin{align} \label{eqn:S_2}
S_2:= 
\sum_{\begin{subarray}{c} \alpha \in A(N) \\ \alpha \equiv v_0\pmod{\mathfrak{m}} \end{subarray}} 
\bigg( \sum_{i=1}^k \beta(\alpha+\mathfrak{h}_i) \bigg) \bigg(\sum_{\begin{subarray}{c} 
\fd_1,\dots,\fd_k: \\ \fd_i \mid (\alpha+\mathfrak{h}_i) \,\forall i \end{subarray}} 
\lambda_{\fd_1,\dots,\fd_k} \bigg)^2,
\end{align}
where the inner sum is a $k$-fold sum over integral ideals and
$\lambda_{\fd_1,\dots,\fd_k}$ are suitably chosen weights to be made explicit later.

\vspace{2mm}
\noindent
Since each summand is non-negative, if we can show that $S_2 > \rho S_1$ 
for some positive $\rho$, then there must be at least one $\alpha\in A(N)$  
such that among $\alpha+\fh_1,\dots,\alpha+\fh_k$ atleast $[\rho]+1$ are
$G_2^K$-numbers.  
We choose the weights $\lambda_{\fd_1,\dots,\fd_k}$ in such a way that 
$\lambda_{\fd_1,\dots,\fd_k}=0$ 
unless $(\fd_i,\mathfrak{m})=1$, $\fd_i$ is square-free, and 
$|\fd_1\cdots \fd_k|\leq R$ 
for each $i=1,\cdots, k$,
where $R$ will be chosen later to be a small power of $N$.  
The main result of this section is the following.

\begin{proposition}
\label{prop:main}
Let $K$ be an imaginary quadratic number field.
Suppose that the primes $\mathcal{P}$ and $G_2^K$-numbers have a common level of distribution 
$0<\vartheta \leq 1$, and set $R=N^{\vartheta}\left(\log N\right)^{-C}$ for some 
constant $C>0$.  
For  a given a piecewise differentiable function $F\colon [0,1]^k \to \mathbb{R}$ supported on the 
simplex $\mathcal{R}_k:=\{(x_1,\dots,x_k)\in[0,1]^k: x_1+\dots+x_k \leq 1\}$, 
we set
\[
\lambda_{\fd_1,\dots,\fd_k} := 
\left(\prod_{i=1}^k \mu(\fd_i)|\fd_i|\right) 
\sum_{\substack{\mathfrak{r}_1,\dots,\mathfrak{r}_k \\ 
\mathfrak{d}_i\mid\mathfrak{r}_i\,\forall i \\ (\mathfrak{r}_i,\mathfrak{m})=1 \,\forall i }} 
\frac{\mu(\mathfrak{r}_1\dots\mathfrak{r}_k)^2}{\prod_{i=1}^k \varphi(\mathfrak{r_i})} 
F\left( \frac{\log|\mathfrak{r}_1|}{\log R},\dots, \frac{\log |\mathfrak{r}_k|}{\log R}\right)
\]
whenever $|\mathfrak{d}_1\dots\mathfrak{d}_k|<R$ and 
$(\mathfrak{d}_1\dots\mathfrak{d}_k,\mathfrak{m})=1$,
and $\lambda_{\mathfrak{d}_1,\dots,\mathfrak{d}_k}=0$ otherwise.\\
Then
\[
S_1 = (1+o(1))\frac{\varphi(\mathfrak{m})^k |A(N)| (c_K \log R)^k}{|\mathfrak{m}|^{k+1}}\widetilde{I}_{1k}(F)
\]
and
\[
S_2 = (1+o(1))\frac{ m_K \varphi(\mathfrak{m})^k |P(N)| (c_K \log R)^{k+1}}{|\mathfrak{m}|^{k+1}} \sum_{m=1}^{k}\left(\widetilde{I}_{2k}^{(m)}(F)+\widetilde{I}_{3k}^{(m)}(F)\right)
\]
where 
$0<\eta \le \frac{\vartheta}{2}$,  $m_K=\frac{\omega_K}{h_K}$ is the Mitsui's constant,
\[
\widetilde{I}_{1k}(F) := \idotsint_{\mathcal{R}_k} F(x_1,\dots,x_k)^2 \,dx_1 \dots dx_k,
\]
\begin{align*}
& \widetilde{I}_{2k}^{(m)}(F):= \\
& \left(\int_{1}^{B/2} \frac{B}{y(B-y)}dy\right)
\left(\idotsint_{\mathcal{R}_{k-1}} 
\left( \int_0^{T_m} F(x_1,\dots,x_k) \, dx_m\right)^2 dx_1 \dots dx_{m-1}
dx_{m+1}\dots dx_k \right)
\end{align*}
and 
\[
\widetilde{I}_{3k}^{(m)}(F)= \int_{B\eta}^1 \frac{B}{y(B-y)}
\idotsint_{\mathcal{R}_{k-1}} \left( \int_0^{T_m(y)} F(x_1,\dots,x_k) \, dx_m\right)^2 dx_1 \dots dx_{m-1}\dots dx_k dy
\]
with $B=2/{\vartheta}$, $T_m=1-x_1-\ldots -x_{m-1}-x_{m+1}-\ldots-x_k$ and 
$T_m(y)=\min (y, T_m)$.
\end{proposition}

\section{Preparations}
\label{preparations}

\noindent
The sum $S_1$ has been  calculated in  \cite [Proposition $2.1$]{CAS}. So we would only
work with $S_2$.
 By squaring innermost sum and interchanging summation from equation (\ref{eqn:S_2}) we can write $S_2$ as

\begin{align}\label{ali:S_2}
S_2:= \sum_{m=1}^{k}S_{2m}=\sum_{m=1}^{k} \sum_{\ul\fa,\ul\fb}\lambda_{\ul\fa}\lambda_{\ul\fb}
 \sum_{\substack{\alpha\in A(N) \\ \alpha\equiv v_0(\mathfrak m)\\ [\fa_j,\fb_j]|(\alpha+\fh_j) \forall j}}
 \beta(\alpha+\fh_m).
 \end{align}

\noindent
We note that $[\fa_i,\fb_i]$ and $[\fa_j,\fb_j]$ are relatively coprime for $i\neq j$
since the primes dividing $\fh_i-\fh_j$ also divides $\mathfrak{m}$. 

\vspace{2mm}
\noindent
If $\beta(\alpha+\fh_m)=1$ then $\alpha +\fh_m=w_1w_2$ with 
$w_1\in \mathcal P(Y',N^b), w_2\in \mathcal P(N^b,\infty)$
where $Y'$ and $N^b$ are as in the definition of $\beta$.
So the norm of $w_2$, $|w_2|=|\sigma(w_2)|^2>N^{2b}>N^{\vartheta}> R$ 
by our choice of $R$ and $b$.  
Hence $\alpha+\fh_m$ has exactly one prime divisor $w_1$ with $|w_1|\leq N^{2b}$.
Since $|\ul\fa|\leq R$, $|\ul\fb|\leq R$ and $\ul\fa, \ul\fb$ are square-free then all prime divisors of $[\ul\fa, \ul\fb]$ 
has norm $\leq R.$

\vspace{2mm}
\noindent
Hence we conclude that either $[\fa_m,\fb_m]=1$ or $[\fa_m,\fb_m]=(w_1).$ Before discussing either of these cases we need the following lemma.
\begin{lemma}\label{lem:shifting error} 
For any function 
$f:\mathcal O_K \rightarrow \mathbb C$ with $|f|\le 1$,
\[ \sum_{\substack{\alpha\in A(N) \\
	   \alpha\equiv \alpha_0 (\mathfrak q)}}
	   f(\alpha+\fh)
	   = \sum_{\substack{\alpha\in A(N) \\
	   \alpha\equiv (\alpha_0+\fh) (\mathfrak q)}}
	   f(\alpha)
	   + O\left( 1+ \left(\frac{|A(N)|}{|\mathfrak q|}
\right)^{1/2}\right).	   
	\]
\end{lemma}

\begin{proof}
Putting $\alpha'=\alpha+\fh$ and 
$\alpha_0'=\alpha_0+\fh$ in the L.H.S, we get
\[
\sum_{\substack{\alpha'\in A(N)+\fh \\
	   \alpha'\equiv \alpha_0' (\mathfrak q)}}
	   f(\alpha').
\]
Since $|f|\le 1$, we get
\[
\sum_{\substack{\alpha'\in A(N)+\fh \\
	   \alpha'\equiv \alpha_0' (\mathfrak q)}}
	   f(\alpha')
=\sum_{\substack{\alpha'\in A(N) \\
	   \alpha'\equiv \alpha_0' (\mathfrak q)}}
	   f(\alpha')  
+O\bigg( 	 \sum_{\substack{\alpha'\in A(N)+\fh
\setminus A(N)\\
	   \alpha'\equiv \alpha_0' (\mathfrak q)}}
	   1 \bigg).
\]
\noindent
Now the $O$-term is
\begin{align*}
	 \sum_{\substack{\alpha'\in A(N)+\fh
\setminus A(N)\\
	   \alpha'\equiv \alpha_0' (\mathfrak q)}}
	   1
&=
 \sum_{\substack{\alpha'\in A(N)+\fh\\
	   \alpha'\equiv \alpha_0' (\mathfrak q)}}
	   1
- \sum_{\substack{\alpha'\in A(N)\\
	   \alpha'\equiv \alpha_0' (\mathfrak q)}}
	   1 
= \frac{|A(N)+\fh|}{|\mathfrak q|}-
   \frac{|A(N)|}{|\mathfrak q|}\\
   & +O\left(\partial(A(N)+\fh,\mathfrak q)\right)
   +O\left(\partial(A(N),\mathfrak q)\right) \ll 1+\left(\frac{|A(N)|}{|\mathfrak q|}
\right)^{1/2}.
\end{align*}
\end{proof}
\noindent

\subsection{}We first suppose that$[\fa_m,\fb_m]=1$.
Replacing $\alpha+\fh_m$ by $\alpha$, the condition $[\fa_j,\fb_j]|(\alpha+\fh_j)$ of 
the inner sum becomes $\alpha \equiv (\fh_m-\fh_j) \text{ modulo } [\fa_j,\fb_j] $ for all $j\neq m$. 
Since $[\fa_j,\fb_j]$ is coprime of $\mathfrak{m}$ for all $j$, by
Chinese remainder theorem, these $k-1$ congruence equations have a common solution $\alpha_0$
 $\pmod{\mathfrak{m}\prod_{j=1}^k [\fa_j,\fb_j]}$ 
where the last product remains unchanged by excluding or including
the index $j=m$ ( as $[\fa_m,\fb_m]=1$).
Using Lemma \ref{lem:shifting error} with $f=\beta$, we get
\[
\sum_{\substack{\alpha\in A(N) \\
	   \alpha\equiv \alpha_0 (\mathfrak q)}}
	   \beta(\alpha+\fh_m)
	   = \sum_{\substack{\alpha\in A(N) \\
	   \alpha\equiv (\alpha_0+\fh_m) (\mathfrak q)}}
	   \beta(\alpha)
	   + O\left( \left(\frac{|A(N)|}{|\mathfrak q|}
\right)^{1/2}\right)
\]
where $\mathfrak q=\mathfrak m \prod_{j=1}^k 
[\mathfrak a_j, \mathfrak b_j]$.
Using this we have 
\begin{align*}
\sum_{\substack{\alpha\in A(N) \\ \alpha\equiv v_0(\mathfrak m)\\ [\fa_j,\fb_j]|(\alpha+\fh_j) \forall j}} \beta(\alpha+\fh_m)=& \sum_{\substack{ \alpha\in A(N) \\ \alpha\equiv \alpha_0' (\mathfrak{q})}}
\beta(\alpha)+O\left( \left(\frac{|A(N)|}{|\mathfrak q|}
\right)^{1/2}\right)\\
=& \frac{1}{\phi(\mathfrak{q})}
\sum_{\substack{ \alpha\in A(N) \\ (\alpha, \mathfrak{q})=1}}\beta(\alpha)
+\mathcal{E}_{\beta}(N, \mathfrak{q}, \alpha_0')
+O\left(\left(\frac{|A(N)|}{|\mathfrak{q}|}\right)^{1/2}\right)\; 
\end{align*}
where $\alpha_{0}^{'}=\alpha_{0}+\fh_m.$

 \vspace{2mm}
\subsection{}                                      
Now assume that $[\fa_m,\fb_m]=(w_1).$ In this case $w_1\in \mathcal{P}(Y', R')$ 
with $R'=R^{1/2}$ because of the support of $\lambda_{\ul\fa}$ and $\lambda_{\ul\fb}$.
Let $\widetilde{w_1}$ be the inverse of $w_1\pmod{\mathfrak{q}/(w_1)}.$
Similarly as above
\begin{align*}
\sum_{\substack{\alpha\in A(N) \\ \alpha\equiv v_0(\mathfrak m)\\ [\fa_j,\fb_j]|(\alpha+\fh_j) \forall j}}
 \beta(\alpha+\fh_m)=\sum_{\substack{\alpha-\fh_m\in A(N) \\ \alpha\equiv \alpha_{o}\left(\mathfrak q/(w_1)\right)}}\beta(\alpha)&=\sum_{\substack{\alpha\in A(N) \\ \alpha\equiv \alpha_{o}\left(\mathfrak q/(w_1)\right)}}\beta(\alpha)-\sum_{\substack{\alpha\in (A(N)+\fh_m)\setminus A(N) \\ \alpha\equiv \alpha_{o}\left(\mathfrak q/(w_1)\right)}}\beta(\alpha)\\
&= \sum_{\substack{\alpha\in A(N) \\ \alpha\equiv \alpha_{o}\left(\mathfrak q/(w_1)\right)}}\beta(\alpha)+O\left( \left(\frac{|A(N)|}{|\mathfrak q|}
\right)^{1/2}\right).
\end{align*}
Now $\alpha\in A(N)$ and $\alpha=w_{1}w_{2}.$   
So we separate  the above sum with respect to primes $w_1$ and $w_2.$ 
We note that 
$w_1w_2\in A(N)$ if and only if $w_2\in A(N/|w_1|^{1/2})$.
Therefore in this case, we have 
 \begin{align*}
 \sum_{\substack{\alpha\in A(N) \\ \alpha\equiv \alpha_{o}\left(\mathfrak q/(w_1)\right)}}\beta(\alpha)=& \sum_{\substack{w_2\in A\left(\frac{N}{|w_{1}|^{1/2}}\right)
\cap \mathcal{P}\\ w_2\equiv \alpha_{o}\widetilde{w_1}\left(\mathfrak {q}/(w_1)\right)}}1 
+ O\left(\left(\frac{|A(N)|}{|\mathfrak{q}|}\right)^{1/2}\right)\\
=&\frac{\pi^{\flat}\left(\frac{N}{|w_{1}|^{1/2}}\right)}{\phi\left(\mathfrak {q}/(w_1)\right)}+\varepsilon\left(\frac{N}{|w_1|^{1/2}},\mathfrak {q}/(w_1),\alpha_{o}\widetilde{w_1}\right)+ O\left(\left(\frac{|A(N)|}{|\mathfrak{q}|}\right)^{1/2}\right).
 \end{align*}
\noindent
For each $\mathfrak{q}$, the number of ways of choosing 
$\mathfrak{a}_1, \dots, \mathfrak{a}_k$ and $\mathfrak{b}_1, \dots, \mathfrak{b}_k$ so that 
$$\mathfrak{m}\prod_{j=1}^{k} [\mathfrak{a}_j, \mathfrak{b}_j] = \mathfrak{q}$$ 
is at most $\tau_{3k}(\mathfrak{q}).$ 
Therefore for each $1\leq m\leq k,$ from equation (\ref{ali:S_2}), the sum $S_{2m}$ can be written as
\begin{align*}
S_{2m}=&\sum_{w_1\in \mathcal P_1^0(Y',N^{b})}\pi^{\flat}\left(\frac{N}{|w_1|^{1/2}}\right)
\sum_{\substack{\ul\fa, \ul\fb \\ [\fa_{m},\fb_{m}]=1\\ \left(w_1, \mathfrak{q}\right)=1}} 
\frac{\lambda_{\ul\fa}\lambda_{\ul\fb}}{\phi\left(\mathfrak m \prod_{j\neq m}[\fa_j,\fb_j]\right)}\\
+&\sum_{w_1\in \mathcal P_1^0(Y', R')}\pi^{\flat}\left(\frac{N}{|w_1|^{1/2}}\right)
\sum_{\substack{\ul\fa, \ul\fb \\ [\fa_{m},\fb_{m}]=(w_1) }} 
\frac{\lambda_{\ul\fa}\lambda_{\ul\fb}}{\phi\left(\mathfrak m \prod_{j\neq m}[\fa_j,\fb_j]\right)}\\
+& O\bigg(\lambda_{\max}^2 \left(\frac{|A(N)|}{|\mathfrak{m}|}\right)^{1/2}
\sum_{\substack{\mathfrak{a}_1,\dots,\mathfrak{a}_k \\ \mathfrak{b}_1,\dots,\mathfrak{b}_k}}  
\frac{1}{\prod_{j=1}^{k} |[\mathfrak{a}_j, \mathfrak{b}_j]|^{\frac{1}{2}}}\bigg)
+O\bigg(\lambda_{\max}^2\sum_{\substack{|\mathfrak{q}|<|\mathfrak{m}| R^2}}
\mu^2(\mathfrak{q})\tau_{3k}(\mathfrak{q}) \varepsilon_{\beta}^{*}\left(N, \mathfrak{q}\right)\bigg)\\
+&O\bigg(\lambda_{\max}^2\sum_{\substack{|\mathfrak{q}|<|\mathfrak{m}|R^2}}\mu^2(\mathfrak{q})
\tau_{3k}(\mathfrak{q})\sum_{w_1\mid \mathfrak{q};w_1\in \mathcal P_1^0(Y', R')}
\varepsilon^{*}\left(\frac{N}{|w_1|^{1/2}},\mathfrak{q}/(w_1)\right)\bigg)
\end{align*}
where $\lambda_{\max}=\sup_{\ul\fa}|\lambda_{\ul\fa}|.$

\vspace{2mm}
\noindent
Using Lemma \ref{lem:merten's theorem on number field}, it can seen that the  first error term of the above expression of $S_{2m}$ 
is bounded above by
\begin{align*} 
 \lambda_{\max}^2 |A(N)|^{1/2} \cdot & \sum_{|\mathfrak{q}| \leq R^2\log\log N} 
 \frac{\mu^2(\mathfrak{q})\tau_{3k}(\mathfrak{q})}{|\mathfrak{q}|^{1/2}} \\
& \leq \lambda_{\max}^2 |A(N)|^{1/2} \cdot R(\log\log N)^1/2 
 \prod_{|\mathfrak{p}| \leq R^2\log\log N} \left(1+\frac{3k}{|\mathfrak{p}|}\right) \\ 
& \ll \lambda_{\max}^2 |A(N)|^{1/2} \cdot R (\log{R})^{3k}.
\end{align*}

\noindent Lemma \ref{L:Lemma4}  gives that the second error term of $S_{2m}$ is bounded above by 
$\lambda_{\max}^2 \frac{|A(N)|}{(\log N)^B}$ for any $B>0.$

Lemma \ref{L:Lemma3} gives that  the third error term of $S_{2m}$ is bounded above by
 \begin{align*}
&\ll \lambda_{\max}^2(\log R)^{2k} \sum_{|\mathfrak{q}|\leq |\mathfrak{m}|R^2}\mu(\mathfrak{q})^2 \tau_{3k}(\mathfrak{q})\sum_{\substack{w_1\mid \mathfrak{q} \\ w_1\in \mathcal{P}_{1}^{0}(Y', R')}}\varepsilon^{*}\left(\frac{N}{|w_1|^{1/2}},\mathfrak{q}/(w_1)\right)\\
&\ll \lambda_{\max}^2(\log R)^{2k} \sum_{|w_1| \leq R}\tau_{3k}(w_1)\sum_{|\mathfrak{s}|\leq \frac{|\mathfrak{m}|R^2}{|w_1|}}\mu(\mathfrak{s})^2 \tau_{3k}(\mathfrak{s})\varepsilon^{*}\left(\frac{N}{|w_1|^{1/2}},\mathfrak{s}\right)\\
&\ll \lambda_{\max}^2 \sum_{|w_1|\leq R}\frac{|A(N)|}{|w_1|\log \left(N/|(w_1)|\right)}
\ll \lambda_{\max}^2 |A(N)|.
\end{align*}
 Combining these estimations of error terms we get the following lemma.

 \begin{lemma}\label{lem:sum S_2m}
  Let $S_{2m}$ be defined as in (\ref{ali:S_2}). Then with the  hypothesis of Proposition \ref{prop:main} we have
  \begin{align*}
S_{2m}=&\sum_{w_1\in \mathcal P_1^0(Y',N^{b})}\pi^{\flat}\left(\frac{N}{|w_1|^{1/2}}\right)\sum_{\substack{\ul\fa, \ul\fb \\ [\fa_{m},\fb_{m}]=1 \\ \left(w_1, \mathfrak{m}\prod_{j=1}^{k}[\fa_{j}, \fb_{j}]\right)=1}} \frac{\lambda_{\ul\fa}\lambda_{\ul\fb}}{\phi\left(\mathfrak m \prod_{j\neq m}[\fa_j,\fb_j]\right)}\\
&+\sum_{w_1\in \mathcal P_1^0(Y', R')}\pi^{\flat}\left(\frac{N}{|w_1|^{1/2}}\right)\sum_{\substack{\ul\fa, \ul\fb \\ [\fa_{m},\fb_{m}]=(w_1) }} \frac{\lambda_{\ul\fa}\lambda_{\ul\fb}}{\phi\left(\mathfrak m \prod_{j\neq m}[\fa_j,\fb_j]\right)} +O \left(\lambda_{\max}^2 |A(N)|\right).
\end{align*}
   \end{lemma}

\vspace{2mm}
\noindent   
 We define
\begin{align} \label{align:S_{2m}(w_1)}
S_{2m}(w_1)=\sum_{\substack{\ul\fa, \ul\fb \\ [\fa_{m},\fb_{m}]\mid(w_1) }} \frac{\lambda_{\ul\fa}\lambda_{\ul\fb}}{\phi\left(\mathfrak m \prod_{j\neq m}[\fa_j,\fb_j]\right)} .
\end{align}
 
\noindent The sum $S_{2m}(w_1)$ is estimated in the following lemma.
 
 \begin{lemma} \label{lem:S_{2m}(w_1)}
 Let $S_{2m}(w_1)$ be defined as in (\ref{lem:S_{2m}(w_1)}). For ideals $\fr_1,\ldots, \fr_k$ of $\mathcal O_K$, we define
\begin{align}\label{align:change of variable depend on primes}
 y_{\fr_1,\ldots,\fr_k}^{(m)}(w_1)
=\prod_{j\neq m} \mu(\fr_j)g(\fr_j)
\sum_{\substack{\ul\fa \\ \fr_j \mid \fa_j \forall j \\ \fa_m\mid (w_1)}} 
\frac{\lambda_{\ul\fa}}{\prod_{j\neq m}\phi(\fa_j)}
\end{align} 
where $g$ is the multiplicative function defined by $g(\mathfrak{p}) = |\mathfrak{p}|-2$ for all prime ideals $\mathfrak{p}$ of $A$.
Let $y_{\max} ^{(m)}(w_1)= \sup_{\mathfrak{r}_1, \dots, \mathfrak{r}_k} |y_{\mathfrak{r}_1, \dots, \mathfrak{r}_k}^{(m)}(w_1)|$. Then we have
\[
 S_{2m}(w_1)=\frac{1}{\phi(\mathfrak{m})}
 \sum_{\substack{\ul\fu \\ \fu_m=1}}\prod_{j\neq m} \frac{\mu^2(\fu_j)}{g(\fu_j)}
 \left(y^{(m)}_{\ul \fu}(w_1)\right)^2 +O\left((y^{(m)}_{\max}(w_1))^2
 (\log R)^{k-1}\frac{(\phi(\mathfrak{m}))^{k-2}}{|\mathfrak{m}|^{k-1}}\frac{1}{D_0}\right).\]
  \end{lemma}
  \begin{proof}
 From the definition of $g$ it follows that
$$
\frac{1}{\varphi([\mathfrak{a}_i, \mathfrak{b}_i])} 
= \frac{1}{\varphi(\mathfrak{a}_i) \varphi(\mathfrak{b}_i)} 
\sum_{\mathfrak{u}_i \mid (\mathfrak{a}_i, \mathfrak{b}_i)} g(\mathfrak{u}_i).
$$
From equation (\ref{align:S_{2m}(w_1)}) we get,
\begin{align*}
&S_{2m}(w_1)=\sum_{\substack{\ul\fa, \ul\fb \\ [\fa_{m},\fb_{m}]\mid(w_1) }} \frac{\lambda_{\ul\fa}\lambda_{\ul\fb}}{\phi\left(\mathfrak m \right)}\prod_{j\neq m}\frac{1}{\phi(\fa_j)\phi(\fb_j)}\sum_{\mathfrak{u_j}\vert (\fa_j,\fb_j)}g(\mathfrak{u_j})\\
&=\frac{1}{\phi(\mathfrak{m})}
\sum_{\substack{\ul\fu \\ \fu_m=1}}\prod_{j\neq m}g(\fu_j)
\sum_{\substack{\ul\fa, \ul\fb \\ [\fa_{m},\fb_{m}]\mid(w_1)\\ \fu_j\vert (\fa_j, \fb_j)\forall j}} 
\frac{\lambda_{\ul\fa}\lambda_{\ul\fb}}{ \prod_{j\neq m}\phi(\fa_j)\phi(\fb_j)}.
\end{align*}
Note that $\lambda_{\fa_1, \ldots, \fa_k}$ is supported on ideals $\fa_1, \ldots, \fa_k$ with 
$(\fa_i, \mathfrak m)=1$ for each $i$ and $(\fa_i, \fa_j)=1 \, \forall i\neq j.$ 
Thus we may drop the requirement that $\mathfrak{m}$ is coprime to each of the $[\fa_i, \fb_i]$ 
from the summation, since these terms have no contribution. Thus the only remaining restriction is that $(\fa_i, \fb_j)=1 \, \forall i\neq j.$ 
So we can remove this coprimality condition by M\"{o}bius inversion to get

\[
 S_{2m}(w_1)
 =\frac{1}{\phi(\mathfrak{m})}
\sum_{\substack{\ul\fu \\ \fu_m=1}}\prod_{j\neq m}g(\fu_j)
\sum_{\fs_{1,2},\ldots, \fs_{k,k-1}}\prod_{\substack{1\le i, j\le k \\ i\neq j}}\mu(\fs_{i,j})
\sum_{\substack{\ul\fa, \ul\fb \\ [\fa_{m},\fb_{m}]\mid(w_1)\\ 
\fu_j\vert (\fa_j, \fb_j)\forall j
\\ \fs_{i,j}|(\fa_i,\fb_j), \forall i\neq j}} 
\frac{\lambda_{\ul\fa}\lambda_{\ul\fb}}{ \prod_{j\neq m}\phi(\fa_j)\phi(\fb_j)}.
\]

\noindent
Now we make the following change of variables:
\[ 
\fc_j=\fu_j \prod_{i\neq j}\fs_{i,j} \text{ and }
\fd_j=\fu_j \prod_{i\neq j}\fs_{j,i}.
\]
\noindent
By using equation (\ref{align:change of variable depend on primes}) 
we can rewrite $S_{2m}(w_1)$ as 
\[
 S_{2m}(w_1)=\frac{1}{\phi(\mathfrak{m})}
 \sum_{\substack{\ul\fu \\ \fu_m=1}}\prod_{j\neq m} \frac{\mu^2(\fu_j)}{g(\fu_j)}
 \sum_{\fs_{1,2},...,\fs_{k,k-1}}\prod_{i\neq j}\frac{\mu(\fs_{i,j})}{g(\fs_{i,j})^2}
 y^{(m)}_{\fc_1,\ldots, \fc_k}(w_1)y^{(m)}_{\fd_1,\ldots, \fd_k}(w_1).
\]
In the above sum $\fs_{i,j}\mid ([\fa_i, \fb_i], [\fa_j, \fb_j])$, hence $\fs_{i,j}$ is coprime to $\mathfrak{m}$
for all $i\neq j$. Then either $\fs_{i,j}=1$ or $|\fs_{i,j}|>D_0$.
For a fixed $i$ and $j,$ the total contribution from the terms with $|\fs_{i,j}|> D_{o}$ is bounded above by
\[
 \frac{(y^{(m)}_{\max}(w_1))^2}{\phi(\mathfrak{m})}
 \Bigg(\sum_{\substack{|\fu|<R \\ (\fu,\mathfrak{m})=1}}\frac{\mu^2(\fu)}{g(\fu)}\Bigg)^{k-1}
 \bigg(\sum_{\fs}\frac{\mu(\fs)^2}{g(\fs)^2}\bigg)^{k(k-1)-1}
 \bigg(\sum_{|\fs_{i,j}|>D_0} \frac{\mu(\fs_{i,j})^2}{g(\fs_{i,j})^2}\bigg).
\]
From Lemma \ref{lem:merten's theorem on number field} the above quantity is bounded above by 
\[
 \frac{(y^{(m)}_{\max}(w_1))^2}{\phi(\mathfrak{m})}
 (\log R)^{k-1}\left(\frac{\phi(\mathfrak{m})}{|\mathfrak{m}|}\right)^{k-1}\frac{1}{D_0}.
\]

\noindent
The main term of $S_{2m}(w_1)$ is obtained from $\mathfrak{s}_{i,j}=1$ for all $i\neq j$
which completes the proof.
\end{proof}

For ideals $\mathfrak{r}_1,\dots,\mathfrak{r}_k$, we define
\[
y_{\mathfrak{r}_1, \dots, \mathfrak{r}_k} =
\bigg(
  \prod_{i=1}^k \mu(\mathfrak{r}_i) \varphi(\mathfrak{r}_i)
\bigg)
\sum_{\substack{\mathfrak{a}_1, \dots, \mathfrak{a}_k \\ \mathfrak{r}_i \mid \mathfrak{a}_i\,\forall i}}
  \frac{\lambda_{\mathfrak{a}_1,  \dots, \mathfrak{a}_k}}{\prod_{i=1}^k |\mathfrak{a}_i|}
\]
and $y_{\max} = \sup_{\mathfrak{r}_1, \dots, \mathfrak{r}_k} |y_{\mathfrak{r}_1, \dots, \mathfrak{r}_k}|$. 

\vspace{2mm}
\noindent
We recall the inversion formula from  \cite [Equation $(5.8)$]{MAY} that
\begin{equation}\label{inversion}
\lambda_{\fa_1,\dots,\fa_k} = 
\bigg(\prod_{i=1}^k \mu(\fa_i)|\fa_i|\bigg) 
\sum_{\substack{\mathfrak{r}_1,\dots,\mathfrak{r}_k \\ 
\mathfrak{a}_i\mid\mathfrak{r}_i\,\forall i }} 
\frac{y_{\fr_1,\ldots, \fr_k}}{\prod_{i=1}^k \varphi(\mathfrak{r_i})}.
\end{equation}
Therefore $\lambda_{\max}\ll y_{\max}(\log R)^{k}$.

\noindent
The following lemma gives a relation between
the quantities $y_{\fr_1,\ldots, \fr_k}^{(m)}(w_1)$ and
$y_{\fr_1,\ldots, \fr_k}$.

\begin{lemma}
 If $\fu_m=1$ ( trivial ideal $\OK$ ), then
 \begin{align*}
 y_{\fu_1,\ldots,\fu_k}^{(m)}(w_1)&=\sum_{\substack{\fr_m\\|\fr_m|\leq R}} 
\frac{y_{\fu_1, \ldots, \fu_{m-1},\fr_m,\fu_{m+1},\ldots, \fu_k}}{\phi(\fr_m)}-\frac{|(w_1)|}{\phi((w_1))}\sum_{\substack{\fs_m \\|\fs_m|\leq R/|w_1|}} 
\frac{y_{\fu_1, \ldots, \fu_{m-1},(w_1)\fs_m,\ldots, \fu_k}}{\phi(\fs_m)}\\
& + O\left(y_{max}\frac{\phi(\mathfrak{m})}{|\mathfrak{m}|}\frac{\log(R/D_{0})}{D_{0}}\right).
 \end{align*}
 \end{lemma}

\begin{proof}
Using (\ref{inversion}), we get
\begin{align*}
 &y_{\fu_1,\ldots,\fu_k}^{(m)}(w_1)
=\prod_{j\neq m} \mu(\fu_j)g(\fu_j)\sum_{\substack{\ul\fa\\\fu_j\mid \fa_j\forall j\\ \fa_m\mid (w_1)}}\frac{1}{\prod_{j\neq m}\phi(\fa_j)}\prod_{i=1}^{k}|\fa_i|\mu(\fa_i)
\sum_{\substack{\ul\fr \\ \fa_i \mid \fr_i \forall i}} 
\frac{y_{\ul\fr}}{\prod_{i=1}^{k}\phi(\fr_i)}\\
&=\prod_{j\neq m} \mu(\fu_j)g(\fu_j)\sum_{\substack{\ul\fr \\ \fu_j \mid \fr_j \forall j}} 
\frac{y_{\ul\fr}}{\prod_{i=1}^{k}\phi(\fr_j)}\sum_{\substack{\ul\fa\mid \ul\fr\\ 
\fu_j\mid \fa_j \forall j \\ \fa_m\mid (w_1)}}\frac{\prod_{j=1}^{k}|\fa_j|\mu(\fa_j)}{\prod_{j\neq m}\phi(\fa_j)}\\
&=\prod_{j\neq m} \mu(\fu_j)g(\fu_j)\sum_{\substack{\ul\fr \\ \fu_j \mid \fr_j \forall j}} 
\frac{y_{\ul\fr}}{\prod_{j=1}^{k}\phi(\fr_j)}\prod_{j\neq m}\frac{|\fu_j|\mu(\fr_j)}{\phi(\fr_j)}\left(1-|(w_1)|\mathds{1}_{(w_1)\mid \fr_m}\right)\\
&=\prod_{j\neq m} \mu(\fu_j)g(\fu_j)\bigg(\sum_{\substack{\ul\fr ; \fu_j \mid \fr_j\\ \fu_m=1}} 
\frac{y_{\ul\fr}}{\prod_{i=1}^{k}\phi(\fr_j)}\prod_{j\neq m}\frac{|\fu_j|\mu(\fr_j)}{\phi(\fr_j)}-|(w_1)|\sum_{\substack{\ul\fr ; \fu_j \mid \fr_j \\ (w_1)\mid \fr_m}} 
\frac{y_{\ul\fr}}{\prod_{j}\phi(\fr_j)}\prod_{j\neq m}\frac{|\fu_j|\mu(\fr_j)}{\phi(\fr_j)}\bigg)
\end{align*}
where $\mathds{1}_{(w_1)\mid \fr_m}$ is the indicator function which takes value $1,$ 
if $(w_1)\mid \fr_m$ and $0$ otherwise.
We see from the support of $y_{\fr_1, \ldots, \fr_k}$ that we may restrict the 
summation over $\fr_j$ to $(\fr_j, \mathfrak{m})=1.$ 
The main term is given by $\fr_j=\fu_j \forall j$, for all other terms
there exists $j\neq m$ such that $|\fr_j|>D_0 |\fu_j|$. 
Therefore the error term is bounded above by
\begin{align*}
y_{\max}\bigg(\prod_{j\neq m} |\fu_j|g(\fu_j)\bigg)
\Bigg(\sum_{\substack{\fr_j>D_{0}\fu_j \\ \fu_j\mid \fr_j}}\frac{\mu^2(\fr_j)}{\phi(\fr_j)^2}\Bigg)&
 \Bigg(|(w_1)|\sum_{\substack{|\fr_m|<R; (w_1)\mid \fr_m \\ (\fr_m, \mathfrak{m})=1}}\frac{\mu(\fr_m)^2}{\phi(\fr_m)}\Bigg)
 \prod_{\substack{1\leq i\leq k\\i\neq j,m}}\bigg(\sum_{\fu_i\mid \fr_i} \frac{\mu(\fr_i)^2}{\phi(\fr_i)^2}\bigg)\\
 &\ll y_{\max} \frac{\phi(\mathfrak{m})}{|\mathfrak{m}|}\frac{\log(R/D_{0})}{D_{0}}.
\end{align*}
The main term given by $\fr_j=\fu_j \forall j\neq m$ is
\begin{align*}
y_{\substack{\fu_1,\ldots,\fu_k \\ \fu_m=1}}^{(m)}(w_1)=\prod_{j\neq m} \frac{|\fu_j|g(\fu_j)}{\phi(\fu_j)^2}\Bigg(\sum_{\fr_m} 
\frac{y_{\fu_1, \ldots, \fu_{m-1},\fr_m,\fu_{m+1},\ldots, \fu_k}}{\phi(\fr_m)}-|(w_1)|\sum_{\substack{\fr_m\\ (w_1)\mid \fr_m}} 
\frac{y_{\fu_1, \ldots, \fu_{m-1},\fr_m,\ldots, \fu_k}}{\phi(\fr_m)}\Bigg).
\end{align*}
Now the proof can be completed by noting that  
$\frac{g(\mathfrak{p})|\mathfrak{p}|}{\phi(\mathfrak{p})^2}=1+O(|\mathfrak{p}|^{-2})$ 
and $\fr_m=(w_1)\fs_m$.

\end{proof}

\section{Choosing the weights}
\label{choosing weights}

 For a real valued piecewise differentiable function $F$ on $\mathcal{R}_k$ as in Proposition \ref{prop:main}
 we define
 $$y_{\fr_1, \ldots, \fr_k}:
 =\mu\left(\prod_{i=1}^{k}\fr_i \right)^2 
 F\left( \frac{\log|\mathfrak{r}_1|}{\log R},\dots, \frac{\log |\mathfrak{r}_k|}{\log R}\right).$$
 \noindent
 Note that, $y_{\fr_1, \ldots, \fr_k}$ is supported on square-free $\fr =\prod_{i=1}^{k}\fr_i$ 
 such that $(\fr,\mathfrak{m})=1.$ Hence
\begin{align*}
y_{\substack{\fu_1,\ldots,\fu_k \\ \fu_m =1}}^{(m)}(w_1)=&\sum_{\substack{|\fr|\leq R \\ \left(\fr, \mathfrak{m}\prod_{j\neq m}\fu_j\right)=1}} 
\frac{\mu(\fr)^2}{\phi(\fr)}F\left( \frac{\log|\mathfrak{u}_1|}{\log R},\dots, \frac{\log |\mathfrak{r}|}{\log R}, \ldots ,\frac{\log |\mathfrak{u}_k|}{\log R}\right)\\
&-\frac{|(w_1)|}{\phi((w_1))}\sum_{\substack{|\fs|\leq R/|w_1| \\ \left(\fs, \mathfrak{m} (w_1)\prod_{j\neq m}\fu_j\right)=1}} 
\frac{\mu(\fs)^2}{\phi(\fs)}F\left( \frac{\log|\mathfrak{u}_1|}{\log R},\dots, \frac{\log |\mathfrak{s}(w_1)|}{\log R}, \ldots ,\frac{\log |\mathfrak{u}_k|}{\log R}\right)\\
&
+ O\left(y_{max}\frac{\phi(\mathfrak{m})}{|\mathfrak{m}|}\frac{\log R}{D_{0}}\right)
=: S_{1}'+S_{2}'+ O\left(y_{max}\frac{\phi(\mathfrak{m})}{|\mathfrak{m}|}\frac{\log R}{D_{0}}\right). 
\end{align*}

\vspace{2mm}
\noindent
\textbf{Estimation of $S_1'$}.
To use Lemma \ref{lem:Dimensional sieve} we set
\[ 
\gamma(\fp)=\begin{cases}
                                             1 & \text{ if } (\fp, \mathfrak{m}\prod_{j\neq m} \fu_j )=1\\
                                             0 & \text{ otherwise} 
                                            \end{cases} \text { and }
h(\fp)=\frac{\gamma(\fp)}{|\fp|-\gamma(\fp)}\; .
\]
and
\[
 F_{\max}=\sup_{t\in [0,1]} ( |F(t)|+|F'(t)|).
\]
The singular series in Lemma \ref{lem:Dimensional sieve} is easily computed to be
\[
\mathfrak{S}
=\frac{\phi(\mathfrak{m}\prod_{j\neq m} \fu_j)}{\mid\mathfrak{m}\prod_{j\neq m} \fu_j \mid}
\]
and also $L\ll \log\log R$. Thus we get
\[
 S_1'=\frac{\phi(\mathfrak{m})}{|\mathfrak{m}|}\prod_{j\neq m} \frac{\phi(\fu_j)}{|\fu_j|}
 c_K(\log R)\int_0^1 F\left( \frac{\log|\fu_1|}{\log R},\ldots, t_m, \ldots,  \frac{\log|\fu_k|}{\log R}\right)dt_m
 +O\left( F_{\max} \log\log R\right)
\]
where $c_K=Res_{s=1}\zeta_K(s)$.

\vspace{2mm}
\noindent
\textbf{Estimation of $S_2'$}.
Observe that 
\[
 \frac{\log(|(w_1)\fs|)}{\log(R/|w_1|)}
 = \frac{\log R}{\log(R/|w_1|)}
 \left( \frac{\log|\fs|}{\log R} + \frac{\log|(w_1)|}{\log R}\right).
\]
Therefore by Lemma \ref{lem:Dimensional sieve}, we get
\begin{align*}
 S_2'=&
 -c_K \frac{\phi(\mathfrak{m})}{|\mathfrak{m}|}\prod_{j\neq m} \frac{\phi(\fu_j)}{|\fu_j|}
 (\log R)\int_{\frac{\log|w_1|}{\log R}}^1 
 F\left( \frac{\log|\fu_1|}{\log R},\ldots, s_m, \ldots,  \frac{\log|\fu_k|}{\log R}\right)ds_m \\
& +O\left( F_{\max} \log\log R\right).
\end{align*}

\noindent
Putting $S_1'$ and $S_2'$ together, we get
\begin{align}\label{weights}
 y_{\substack{\fu_1,\ldots,\fu_k \\ \fu_m =1}}^{(m)}(w_1)
 &= \frac{\phi(\mathfrak{m})}{|\mathfrak{m}|}c_K\prod_{j\neq m} 
 \frac{\phi(\fu_j)}{|\fu_j|}
(\log R) \left( F^{(m)}_{\fu_1,\ldots, \fu_k}- 
F^{(m)}_{\fu_1,\ldots, \fu_k}(w_1)\right)\\ \nonumber
&+O\left( F_{\max} \log\log R\right) 
+ O\left(F_{\max}\frac{\phi(\mathfrak{m})}{|\mathfrak{m}|}\frac{\log R}{D_{0}}\right)
\end{align}
where
\[
 F^{(m)}_{\fu_1,\ldots, \fu_k}
 =\int_0^1 F\left( \frac{\log|\fu_1|}{\log R},\ldots, t_m, \ldots,  \frac{\log|\fu_k|}{\log R}\right)dt_m
\]
and
\[
F^{(m)}_{\fu_1,\ldots, \fu_k}(w_1)=
\int_{\frac{\log|w_1|}{\log R}}^1 
 F\left( \frac{\log|\fu_1|}{\log R},\ldots, s_m,\ldots ,  \frac{\log|\fu_k|}{\log R}\right)ds_m.
\]

\section{Proof of Proposition \ref{prop:main}}
\label{Proof: main}

Using the value of  $y_{\substack{\fu_1,\ldots,\fu_k }}^{(m)}(w_1)$ given by 
equation (\ref{weights}) in Lemma \ref{lem:S_{2m}(w_1)}
we get
\begin{align*}
 S_{2m}(w_1)= &
 \frac{1}{\phi(\mathfrak{m})}
 \sum_{\substack{\ul\fu \\ \fu_m=1}}
 \prod_{j\neq m}\frac{\mu(\fu_j)^2}{g(\fu_j)}
\left(\frac{\phi(\mathfrak{m})}{|\mathfrak{m}|}c_K\log R
\prod_{j\neq m}\frac{\phi(\fu_j)}{|\fu_j|}
 \left(F^{(m)}_{\fu_1,\ldots, \fu_k}- F^{(m)}_{\fu_1,\ldots, \fu_k}(w_1)\right)\right)^2\\
& + O\left( (F_{\max})^2
 (\log R)^{k+1}\frac{(\phi(\mathfrak{m}))^{k}}{|\mathfrak{m}|^{k+1}}\frac{1}{D_0}\right).
\end{align*}

Setting $Y':=Y^{1/2} $ and the above equation in Lemma \ref{lem:sum S_2m}, we have
\begin{align*}
S_{2m} &=\frac{\phi(\mathfrak{m})}{|\mathfrak{m}|^2}c_K^2 (\log R)^2 
\sum_{\substack{w_1\\Y\leq |w_1|\leq R}}\pi^{\flat}\left(\frac{N}{|w_1|^{1/2}}\right)
\sum_{\substack{\ul\fu \\ \fu_m=1}}\prod_{j\neq m}\frac{\phi(\fu_j)^2}{g(\fu_j)|\fu_j|^2}
\left(\widetilde{F}_{\ul\fu}^{(m)}(w_1)\right)^2 \\
&+ \frac{\phi(\mathfrak{m})}{|\mathfrak{m}|^2}c_K^2 (\log R)^2 
\sum_{\substack{w_1\\R<|w_1| \leq N^{2b}}}\pi^{\flat}\left(\frac{N}{|w_1|^{1/2}}\right)
\sum_{\substack{\ul\fu \\ \fu_m=1}}\prod_{j\neq m}\frac{\phi(\fu_j)^2}{g(\fu_j)|\fu_j|^2}
\left(F_{\ul\fu}^{(m)}\right)^2\\
&+O\bigg((F_{\max})^2
 (\log R)^{k+1}\frac{(\phi(\mathfrak{m}))^{k}}{|\mathfrak{m}|^{k+1}}\frac{1}{D_0}\sum_{w_1\in \mathcal{ P}_1^0(Y',N^b)}\pi^{\flat}\left(\frac{N}{|w_1|^{1/2}}\right) \bigg)
 +O\left(F_{\max}^2|A(N)|\right)
\end{align*}
where
\[
\widetilde{F}^{(m)}_{\fu_1,\ldots, \fu_k}(w_1)=
\int_{0}^{\frac{\log|w_1|}{\log R}} 
 F\left( \frac{\log|\fu_1|}{\log R},\ldots, t_m, \ldots,  \frac{\log|\fu_k|}{\log R}\right)dt_m .
\]
Using Lemma \ref{lem:PNT2} we can say that 
$\pi^{\flat}\left(\frac{N}{|w_1|^{1/2}}\right)=|\mathcal{P}(N)|
\frac{\alpha(|w_1|)}{|w_1|}+O_{K}\left(\frac{N^2}{|w_1|(\log N)^2}\right)$, 
where $\alpha(u):= \frac{\log (N^2)}{\log \left({N^2}/u\right)}.$\\
Using this the main term of $S_{2m}$ becomes
\begin{align*}
&=\frac{\phi(\mathfrak{m})}{|\mathfrak{m}|^2}c_K^2 (\log R)^2 |\mathcal{P}(N)|
\sum_{\substack{w_1\\Y\leq |w_1|\leq R}}\frac{\alpha(|w_1|)}{|w_1|}
\sum_{\substack{\ul\fu \\ \fu_m=1}}\prod_{j\neq m}\frac{\phi(\fu_j)^2}{g(\fu_j)|\fu_j|^2}
\left(\widetilde{F}_{\ul\fu}^{(m)}(w_1)\right)^2 \\
&+ \frac{\phi(\mathfrak{m})}{|\mathfrak{m}|^2}c_K^2 (\log R)^2 |\mathcal{P}(N)|
\sum_{\substack{w_1\\R< |w_1|\leq N^{2b}}}\frac{\alpha(|w_1|)}{|w_1|}
\sum_{\substack{\ul\fu \\ \fu_m=1}}\prod_{j\neq m}\frac{\phi(\fu_j)^2}{g(\fu_j)|\fu_j|^2}
\left(F_{\ul\fu}^{(m)}\right)^2 \\
& =:\frac{\phi(\mathfrak{m})}{|\mathfrak{m}|^2}c_K^2 (\log R)^2 |\mathcal{P}(N)| 
\left(\sum_{\substack{w_1\\Y\leq |w_1|\leq R}}\frac{\alpha(|w_1|)}{|w_1|}S_4
+\sum_{\substack{w_1\\R< |w_1|\leq N^{2b}}}\frac{\alpha(|w_1|)}{|w_1|}S_5\right). 
\end{align*}
\textbf{Estimation of $S_4$ and $S_5$}.
To calculate both $S_4$ and $S_5$ we use Lemma \ref{lem:Dimensional sieve} 
with 
\[
h(\mathfrak{p})=\frac{\gamma(\mathfrak{p})}{|\mathfrak{p}|-\gamma(\mathfrak{p})}
\text{ and } \gamma(\fp)=\begin{cases}
                                             1-\frac{|\mathfrak{p}|^2-3|\mathfrak{p}|+1}{|\mathfrak{p}|^3-
                                             |\mathfrak{p}|^2-2|\mathfrak{p}|+1} & \text{ if } (\fp, \mathfrak{m})=1\\
                                             0 & \text{ otherwise.} 
                                            \end{cases}
\]
The singular series can be easily computed to be 
$\mathfrak{S}=\frac{\phi(\mathfrak{m})}{\mid\mathfrak{m}\mid}+O\left(\frac{\phi(\mathfrak{m})}{\mid\mathfrak{m}\mid D_{0}}\right)$ and also $L\ll \log D_{0}.$
Recalling the coprimality conditions $S_4$ can be written as 
\begin{align*}
S_4= \sum_{\substack{\fu_1, \ldots, \fu_{m-1},\fu_{m+1}, \ldots, \fu_k \\ \left(\fu_i,\fu_j\right)=1 \forall i\neq j\\ \left(\fu_j, \mathfrak{m}\right)=1}}\prod_{j\neq m}\frac{\phi(\fu_j)^2}{g(\fu_j)|\fu_j|^2}
\left(\widetilde{F}_{\ul\fu}^{(m)}(w_1)\right)^2 .
\end{align*}
We note that, two ideals $\fa$ and $\fb$ with $(\fa, \mathfrak{m})=(\fb, \mathfrak{m})=1$ 
but $(\fa, \fb)\neq 1$ must have a common prime factor with norm greater than $D_{0}.$ Thus we can drop the requirement that $\left(\fu_i, \fu_j\right)=1,$ at the cost of an error of size
\begin{align*}
&\ll F_{\max}^2 \sum_{|\mathfrak{p}|>D_{0}}\sum_{\substack{\mid\fu_1\mid, \ldots, \mid\fu_{m-1}\mid,\mid\fu_{m+1}\mid, \ldots, \mid\fu_k\mid <R \\ \mathfrak{p}\mid \fu_i, \fu_j\forall i\neq j\\ \left(\fu_j, \mathfrak{m}\right)=1}}
\prod_{j\neq m}\frac{\mu(\fu_j)^2\phi(\fu_j)^2}{g(\fu_j)|\fu_j|^2}\\
&\ll F_{\max}^2 \bigg(\sum_{|\mathfrak{p}|>D_{0}}\frac{\phi(\mathfrak{p})^4}{g(\mathfrak{p})^2|\mathfrak{p}|^4}\bigg)
\bigg(\sum_{\substack{\mid\fr\mid <R \\ \left(\fr, \mathfrak{m}\right)=1}}
\frac{\mu(\fr)^2\phi(\fr)^2}{g(\fr)|\fr|^2}\bigg)^{k-1}
\ll F_{\max}^2 \left(\frac{\phi(\mathfrak{m})}{|\mathfrak{m}|}\right)^{k-1}\frac{(\log R)^{k-1}}{D_{0}}.
\end{align*} 
Thus it is enough to evaluate the following sums.
\[
\sum_{\substack{\fu_1, \ldots, \fu_{m-1},\fu_{m+1}, \ldots, \fu_k \\  \left(\fu_j, \mathfrak{m}\right)=1}}\prod_{j\neq m}\frac{\phi(\fu_j)^2}{g(\fu_j)|\fu_j|^2}
\left(\widetilde{F}_{\ul\fu}^{(m)}(w_1)\right)^2 \text{ and }
\sum_{\substack{\fu_1, \ldots, \fu_{m-1},\fu_{m+1}, \ldots, \fu_k \\ \left(\fu_j, \mathfrak{m}\right)=1}}\prod_{j\neq m}\frac{\phi(\fu_j)^2}{g(\fu_j)|\fu_j|^2}
\left(F_{\ul\fu}^{(m)}\right)^2 .
\]
Using Lemma \ref{lem:Dimensional sieve} we have,
\begin{align*}
S_4=\left(\frac{\phi(\mathfrak{m})}{|\mathfrak{m}|}\right)^{k-1} c_K^{k-1}(\log R)^{k-1} I_{3k}^{(m)}(F(w_1)) +O\left(F_{\max}^2 \left(\frac{\phi(\mathfrak{m})}{|\mathfrak{m}|}\right)^{k-1}\frac{1}{D_{0}}(\log R)^{k-1}\right)
\end{align*}
and 
\begin{align*}
S_5=\left(\frac{\phi(\mathfrak{m})}{|\mathfrak{m}|}\right)^{k-1} c_K^{k-1}(\log R)^{k-1} I_{2k}^{(m)}(F) +O\left(F_{\max}^2 \left(\frac{\phi(\mathfrak{m})}{|\mathfrak{m}|}\right)^{k-1}\frac{1}{D_{0}}(\log R)^{k-1}\right)
\end{align*}
where
\[
I_{3k}^{(m)}(F(w_1))= \idotsint_{\mathcal{R}_{k-1}} \left( \int_0^{T_m\left(\frac{\log |w_1|}{\log R}\right)} F(x_1,\dots,x_k) \, dx_m\right)^2 dx_1 \dots dx_{m-1}dx_{m+1}\dots dx_k
\]
and 
\[
I_{2k}^{(m)}(F)= \idotsint_{\mathcal{R}_{k-1}} \left( \int_0^{T_m} F(x_1,\dots,x_k) \, dx_m\right)^2 dx_1 \dots dx_{m-1}dx_{m+1}\dots dx_k 
\]
where $T_m$ and $T_m(y)$ are as defined in the statement of Proposition \ref{prop:main}.
We note that the integral $I_{2k}^{(m)}(F)$ is independent of prime element
$w_1$ of $\mathcal{O}_K$. Using the estimations  of the sums $S_4$ and $S_5$ 
the term $S_{2m}$ becomes
\begin{align*}
S_{2m}&=\\
&\frac{\phi(\mathfrak{m})^k}{|\mathfrak{m}|^{k+1}}(c_K\log R)^{k+1} |\mathcal{P}(N)|\biggl(
\sum_{\substack{w_1\\ Y\leq |w_1|\leq R}}\frac{\alpha(|w_1|)}{|w_1|}I_{3k}^{(m)}(F(w_1))
+ I_{2k}^{(m)}(F) \sum_{\substack{w_1\\R< |w_1|\leq N^{2b}}}\frac{\alpha(|w_1|)}{|w_1|}\biggr)\\
&+O\biggl((F_{\max})^2
 (\log R)^{k+1}\frac{\phi(\mathfrak{m})^{k}}{|\mathfrak{m}|^{k+1}}\frac{1}{D_0}\sum_{Y< |w_1|\leq N^{2b}}\pi^{\flat}\left(\frac{N}{|w_1|^{1/2}}\right)\biggr).
\end{align*}
Finally it remains to calculate the following sums
\[
S_6:= \sum_{\substack{w_1\\R< |w_1|\leq N^{2b}}}\frac{\alpha(|w_1|)}{|w_1|} \text{ and }
S_7:= \sum_{\substack{w_1\\Y\leq |w_1|\leq R}}\frac{\alpha(|w_1|)}{|w_1|} \left(V^{(m)}\left(\frac{\log |w_1|}{\log R}\right)\right)^2
\]
where $V^{(m)}(y):=\int_{0}^{y}F(x_1, \ldots, x_k)dx_m .$

\vspace{2mm}
\noindent
Using Lemma \ref{lem:PNT1},  we have
\begin{align*}
 &\sum_{w\in \mathcal {P}^{0}(u^{1/2})} \log |w|
 =m_K u+ E(u)\ \text{ where } E(u)=O_K \left( \frac{u}{\log u}\right)
\end{align*}
where $m_K=\frac{\omega_K}{h_K}.$
\noindent
From the above estimations, we get
\begin{align*}
 S_7 &=m_K
 \int_{Y}^R \alpha(u)\left( V^{(m)}\left(\frac{\log u}{\log R}\right)\right)^2 \frac{du}{u\log u}
+ \int_{Y}^R \alpha(u)\left( V^{(m)}\left(\frac{\log u}{\log R}\right)\right)^2 \frac{dE(u)}{u\log u}\\
&:=S_8+S_9.
 \end{align*}
Putting  $u=R^y$, first integral $S_8$ gives main term for $S_7$ which is
\[
S_8=m_K \int_{B'\eta}^1 \frac{B'}{y(B'-y)} \left(V^{(m)}(y)\right)^2 dy
\]
where
\[ 
\eta=\frac{\log Y}{\log (N^2)}, \  \ B'=\frac{\log (N^2)}{\log R}.
 \]
Second integral $S_9$ giving error term of $S_7$ can be estimated as
\[ 
 S_9\ll (\log R)^{-1}.
\]
Since $R=N^{\vartheta}(\log N)^{-C}$ 
we observe that 
$$B'=B+O\left(\frac{\log \log N}{\log N}\right)$$
where $B=\frac{2}{\vartheta}$ as defined in Proposition \ref{prop:main}.
Therefore combining these estimations we have
\[
S_7= m_K\int_{B\eta}^1 \frac{B}{y(B-y)} \left(V^{(m)}(y)\right)^2 dy 
+O\left(\frac{\log \log N}{\log N}\right).
\]
By using the same method we have
\[
S_6= m_K\int_{1}^{B/2} \frac{B}{y(B-y)} dy 
+O\left(\frac{\log \log N}{\log N}\right).
\]
Therefore we conclude that
\begin{align}\label{Assymp:S2m}
S_{2m}=& m_K\frac{\phi(\mathfrak{m})^k}{|\mathfrak{m}|^{k+1}}(c_K\log R)^{k+1} 
|\mathcal{P}(N)|
\left(\widetilde{I}_{2k}^{(m)}(F)+\widetilde{I}_{3k}^{(m)}(F)\right)\\ \nonumber
+& O\left(F_{\max}^2(\log R)^{k+1}\frac{(\phi(\mathfrak{m}))^{k}}{|\mathfrak{m}|^{k+1}}
\frac{1}{D_0} |\mathcal{P}(N)|\right)
\end{align}

 Recall that 
 \[S_2=\sum_{1\leq m\leq k} S_{2m}.\] 
Therefore Proposition \ref{prop:main} follows from \ref{Assymp:S2m}.

\section{proof of theorem \ref{main theorem}, corollary \ref{corollary: gap 2 number field} and corollary \ref{corollary: gap 8 number fields}}
\label{proofs}
\noindent
We start with the following corollary 
of the Proposition \ref{prop:main}.
\begin{corollary}\label{application of main proposition}
Let $K$ be an imaginary quadratic number field and $m_K$ be its Mitsui constant.
Suppose that the primes $\mathcal{P}$ and $G_2^K$-numbers have a common level of 
distribution $0<\vartheta \leq 1$. Let $(\mathfrak{h}_1, \ldots, \mathfrak{h}_k)\in 
\mathcal{O}_K^k$ be an admissible tuple. 
Let $B, \tilde{I}_{1k}(F), \tilde{I}_{2k}^{(m)}(F)$ and $\tilde{I}_{3k}^{(m)}(F)$, 
$1\leq m\leq k$ be defined as in the statement of Proposition \ref{prop:main}. 
Let $\mathcal{S}_k$ denote the set of piecewise differentiable functions 
$F: [0, 1]\to \mathbb{R}$ supported on $\mathcal{R}_k$ such that  
$\tilde{I}_{1k}(F), \tilde{I}_{2k}^{(m)}(F)$ and $\tilde{I}_{3k}^{(m)}(F)$, 
are non-zero for all $m$ in $1\le m\le k$. Let
\begin{align*}
\tilde{M}_k:=\sup_{F\in \mathcal{S}_k}\frac{\sum_{m=1}^{k}(\tilde{I}_{2k}^{(m)}(F)+
\tilde{I}_{3k}^{(m)}(F))}{\tilde{I}_{1k}(F)} \quad \text{ and } 
\quad \tilde{r}_k:= \bigg\lceil \frac{m_K\tilde{M}_k}{B}\bigg\rceil
\end{align*}
Then there are infinitely many $\alpha\in \mathcal{O}_K$ such that at least $\tilde{r}_k$ of the $\alpha+\mathfrak{h_1}, \ldots, \alpha+\mathfrak{h}_k$ are $G_2^{K}$-numbers.
\end{corollary}
\begin{proof}
Since each summand is non-negative, if $S:=S_2 - \rho S_1>0$ 
for some positive $\rho$, then there is an $\alpha\in A(N)$  
such that $\alpha+\fh_1,\dots,\alpha+\fh_k$ contains atleast $[\rho]+1$  
$G_2^K$-numbers. 
Therefore it is enough to show that $S>0$ for all sufficiently large $N$.

Fix a $\delta>0$ and $0<\epsilon<\frac{\delta B}{m_K}$, 
then choose $\tilde{F}\in \mathcal{S}_K$ so that 
\[
\sum_{m=1}^{k}(\tilde{I}_{2k}^{(m)}(\tilde{F})+\tilde{I}_{3k}^{(m)}(\tilde{F}))>(\tilde{M}_k-\epsilon)\tilde{I}_{1k}(\tilde{F}).
\] 
Using Proposition \ref{prop:main}, we obtain
\begin{align*}
S& =(1+o(1)) \frac{(\varphi(\mathfrak{m})^k |A(N)| (c_K \log R)^{k}}
{|\mathfrak{m}|^{k+1}} \bigg(\frac{m_K}{B}\sum_{m=1}^{k}(\tilde{I}_{2k}^{(m)}(F)+\tilde{I}_{3k}^{(m)}(F))-\rho \tilde{I}_{1k}(F)\bigg)\\
& \geq \frac{(\varphi(\mathfrak{m})^k |A(N)| (c_K \log R)^{k}}
{|\mathfrak{m}|^{k+1}}\tilde{I}_{1k}(\tilde{F}) 
\bigg(\frac{m_K}{B}(\tilde{M}_k -\epsilon)-\rho \bigg).
\end{align*}
If $\rho=\frac{ m_K\tilde{M}_k}{B}-\delta$, then 
$S>0$ for large $N$. Since $\delta$ is arbitrary, there are 
infinitely many $\alpha \in \mathcal{O}_K$ such that at least 
$\big\lceil \frac{m_K\tilde{M}_K}{B}\big\rceil$ of the 
$\alpha+\mathfrak{h}_1, \ldots, \alpha+\mathfrak{h}_k$ are $G_2^K$-numbers.
\end{proof}

To complete the proof of the Theorem \ref{main theorem} it is enough to show that
$\tilde{r}_k \rightarrow \infty$ as $k \rightarrow \infty$.
Since the integrals $\tilde{I}_{3k}^{(m)}(F)$ are positive
\[
 \tilde{M}_k\ge \sup_{F\in \mathcal{S}_k}
 \frac{\sum_{m=1}^{k}\tilde{I}_{2k}^{(m)}(F)}{\tilde{I}_{1k}(F)} 
\]
It follows from Section 7 of \cite{MAY} that 
\[
\sup_{F\in \mathcal{S}_k}
 \frac{\sum_{m=1}^{k}\tilde{I}_{2k}^{(m)}(F)}{\tilde{I}_{1k}(F)} 
 >c\log k
\]
for sufficiently large $k$ and an absolute constant $c>0$
( note that $\tilde{I}_{2k}^{(m)}(F)$ is a
positive constant multiple of $J_k^{(m)}(F)$ in \cite{MAY}). 
This completes the proof as $\tilde{r}_k$ is
directly proportional to $\tilde{M}_k$.

\begin{remark}
 Comparing the integral $\tilde{I}_{2k}^{(m)}(F)$ with $J_k^{(m)}(F)$ as in
 \cite{MAY}, we can show that $\tilde{M}_k\ge (1.0986)M_k$ where $M_k$ is as in 
 \cite{MAY}. Since $\omega_K=2$ for fields with class number more than 2,
 From Proposition 4.3 of \cite{MAY}, it follows that
 \[
  \tilde{r}_k> \frac{1.0986}{2h_K}(\log k -2\log\log k -2)
 \]
for sufficiently large $k$. 
We conclude that there exists infinitely many $\alpha\in\mathcal{O}_K$ such that
for any admissible $k$-tuple $(\fh_1,\cdots,\fh_k)$ there is atleast one
$G_2^K$-number among $\alpha+\fh_1,\cdots, \alpha+\fh_k$ 
(i.e $\tilde{r}_k\ge 1$) provided
$$\log k -2\log\log k \ge (1.82)h_K+2$$
and in that case the gap is bounded above by $\fh_k-\fh_1$ where
$(\fh_1,\cdots,\fh_k)$ is an admissible $k$-tuple.
Therefore gap between $G_2^K$-numbers are bounded in terms of class numbers. 
 
\end{remark}

\vspace{.5cm}
\noindent
To prove the Corollary \ref{corollary: gap 2 number field}  stated in Section $1$, we need following lemmas .

\begin{lemma}[Proposition 3.1, \cite{CAS}]
\label{lem:admissible in number field} Suppose that $\mathcal{H}$ is an admissible tuple in $\mathbb{Z}$. Then $\mathcal{H}$ is also an admissible tuple in $\mathcal{O}_K$ for every number field $K$.
\end{lemma} 
\noindent
Using Proposition \ref{Bombieri vinogradov theorem on number fields}, we obtain the following lemma.
\begin{lemma}[Corollary 1.4, \cite{DAR}]
\label{LOD}
Let $K$ be an imaginary quadratic field.
Then product of two primes in $\OK$  have level of distribution $\frac{1}{2}$.
\end{lemma}
\begin{proof}[\bf{Proof of Corollary \ref{corollary: gap 2 number field}}]
Recall that
\begin{align}
S=S_2-\rho S_1
\end{align}\label{align:sum S}
where $S_2$ and $S_1$ are defined as in Proposition \ref{prop:main}. 
We choose $F(t_1, \ldots, t_k)$ to be a symmetric polynomials in $t_1, \ldots, t_k.$
By Proposition \ref{prop:main},  we see that
\begin{align*}
S=(1+o(1)) \frac{(\varphi(\mathfrak{m})^k |A(N)| (c_K \log R)^{k}}
{|\mathfrak{m}|^{k+1}}\widetilde{I}
\end{align*}
where 
\[
\widetilde{I}=\frac{m_K k}{B}\left(\widetilde{I}_{2k}^{(1)}(F)+\widetilde{I}_{3k}^{(1)}(F)\right)-\rho \widetilde{I}_{1k}(F).
\]
We know that
\[ 
\omega_{K_d}=\begin{cases}
                                             4 & \text{ if } d=-1\\
                                             6 & \text{ if } d=-3\\
                                             2 & \text{ otherwise}. 
                                            \end{cases} 
\]
 Therefore Mitsui's constant for imaginary quadratic number fields of class number one are listed below:
 \[ 
m_{K_d}=\begin{cases}
                                             4 & \text{ if } d=-1\\
                                             6 & \text{ if } d=-3\\
                                             2 & \text{ otherwise}. 
                                            \end{cases} 
\]
 
 \noindent
 For Corollary \ref{corollary: gap 2 number field}, we take $k=2, \vartheta=\frac{1}{2}, \rho=1, \eta =\frac{1}{200}, F(t_1, t_2)=1-F_1(t_1, t_2)+F_2(t_1, t_2),$ where $F_1(t_1, t_2)=t_1+t_2$ and $F_2(t_1, t_2)=t_1^2+t_2^2$.

\vspace{2mm}
\noindent
Using SageMath we obtain
\begin{align*}
\widetilde{I}_{12}(F)=0.227778, 
\widetilde{I}_{22}^{(1)}(F)=0.169151, 
\widetilde{I}_{32}^{(1)}(F)=0.150712.
\end{align*}

\noindent
Therefore we have $\widetilde{I}>0 .$
 Hence  Corollary \ref{corollary: gap 2 number field} follows 
 from  Lemma \ref{lem:admissible in number field} 
 considering admissible set $\{0, 2\}$ 
 and invoking Lemma \ref{LOD}.
 
 \end{proof}
 
 \begin{proof}[\bf{Proof of Corollary \ref{corollary: gap 8 number fields}}]
 For imaginary quadratic number field $K_d$ of class number two
 \[
 \omega_{K_d}=2 \quad \text{ and }  \quad m_{K_d}=1.
 \]
 For Corollary \ref{corollary: gap 8 number fields}, we choose $k=4, \vartheta=\frac{1}{2}, \rho=1, \eta =\frac{1}{200}, F(t_1, t_2, t_3, t_4)=1-F_1(t_1, t_2, t_3, t_4)+F_2(t_1, t_2, t_3, t_4),$ where $F_1(t_1, t_2, t_3, t_4)=t_1+t_2+t_3+t_4$ and $F_2(t_1, t_2, t_3, t_4)=t_1^2+t_2^2+t_3^2+t_4^2$.
 
 \vspace{2mm}
 \noindent
 Using SageMath we obtain
\begin{align*}
\widetilde{I}_{14}(F)=0.0095238, 
\widetilde{I}_{24}^{(1)}(F)=0.0044928, 
\widetilde{I}_{34}^{(1)}(F)=0.0059492.
\end{align*}
\noindent
Therefore we have $\widetilde{I}>0 .$
 Hence  Corollary \ref{corollary: gap 8 number fields} follows from  
 Lemma \ref{lem:admissible in number field} considering admissible set
 $\{0,2,6,8\}$ and using
 Lemma \ref{LOD}.

 \end{proof}

\end{document}